\documentclass[12pt,reqno]{amsart}
\usepackage[utf8]{inputenc}
\usepackage{amsmath}
\usepackage{amssymb}
\usepackage{amsthm}
\usepackage{xcolor}
\usepackage{fourier}
\usepackage{fullpage}
\usepackage{hyperref}
\usepackage{tikz}
\usepackage{soul}
\usepackage{mathrsfs}
\usepackage{enumitem}
\usepackage{bm}
\usepackage{mathtools}

\newcommand{\fg}{{\mathfrak g}}

\newcommand{\fv}{{\mathfrak v}}
\newcommand{\G}{{\mathbb{G}}}
\newcommand{\ve}{{\varepsilon}}
\newcommand{\R}{{\mathbb{R}}}
\newcommand{\bW}{\mathbb{W}}
\newcommand{\bV}{\mathbb{V}}
\newcommand{\zero}{\mathbf{0}}
\DeclareMathOperator{\diam}{diam}
\newcommand{\0}{{\bf 0}}

\newcommand{\stm}{\setminus}
\newcommand{\supp}{\operatorname{supp}}

\newtheorem{thm}{Theorem}[section]
\newtheorem{theorem}{Theorem}[section]

\newtheorem{lemma}[thm]{Lemma}

\newtheorem{corollary}[thm]{Corollary}
\theoremstyle{definition}

\newtheorem{proposition}[thm]{Proposition}
\theoremstyle{definition}

\theoremstyle{definition}
\newtheorem{definition}[thm]{Definition}
\theoremstyle{definition}

\newtheorem{example}[thm]{Example}
\newtheorem{remark}[thm]{Remark}

\newtheorem*{roadmap}{Roadmap}

\numberwithin{equation}{section}

\title{Singular integrals on $C^{1,\alpha}$ intrinsic graphs in step 2 Carnot groups}

\author{Vasileios Chousionis}
\address{Department of Mathematics, University of Connecticut}
\email{vasileios.chousionis@uconn.edu}

\author{Sean Li}
\address{Department of Mathematics, University of Connecticut}
\email{sean.li@uconn.edu}

\author{Lingxiao Zhang}
\address{Department of Mathematics, University of Connecticut}
\email{lingxiao.zhang@uconn.edu}

\thanks{V.~C.\ was supported by Simons Foundation Collaboration grant 521845 and by  NSF grant 2247117. L.~Z. is supported by AMS-Simons travel grant. }

\begin{document}

\begin{abstract}
We study singular integral operators induced by Calder\'on-Zygmund kernels in any step-$2$ Carnot group $\mathbb{G}$. We show that if such an operator satisfies some natural cancellation conditions then it is  $L^2$ bounded on all intrinsic graphs of $C^{1,\alpha}$ functions over vertical hyperplanes that do not have rapid growth at $\infty$.

In particular, the result applies to the Riesz operator $\mathcal{R}$ induced by the kernel
$$
\mathsf{R}(z)= \nabla_{\mathbb{G}} \Gamma(z), \quad z\in \mathbb{G}\backslash \{0\},
$$
the horizontal gradient of the fundamental solution of the sub-Laplacian. The $L^2$ boundedness of $\mathcal{R}$ is connected with the question of removability for Lipschitz harmonic functions. As a corollary of our result, we infer that closed subsets with positive $(Q-1)$-Hausdorff measure (where $Q$ is the homogeneous dimension of $\G$) of the intrinsic graphs mentioned above are non-removable. 
\end{abstract}

\maketitle

\section{Introduction}

The study of singular integral operators (SIOs) on Lipschitz graphs and other lower dimensional subsets of Euclidean spaces has been a central topic of investigation at the intersection of harmonic analysis and geometric measure theory; see e.g.  \cite{cal77,cmm,mmv,david88,dsbook,DS1,tolsabook}. Progress in this area has been driven in part by the significance of certain singular integrals in complex analysis, potential theory, and partial differential equations. Notably, the Cauchy transform and the $1$-codimensional Riesz transform play a key role in analyzing the removability of singularities for bounded analytic and Lipschitz harmonic functions.

A closed set \(E \subset \mathbb{R}^n\) is said to be removable for Lipschitz harmonic functions (RLH), if for every open set \(D\) and for every Lipschitz function \(f: D \to \mathbb{R}\) which is harmonic in \(D \setminus E\), \(f\) must be harmonic on all of \(D\). The characterization of  RLH sets is strongly dependent on the \((n - 1)\)-dimensional Riesz transform $\mathsf{R}_{n-1}$. This is the singular integral whose kernel is the gradient of the fundamental solution of the Laplacian, i.e.
\[
K_{n-1}(x) = \nabla \Phi_n(x),
\]
where 
$$\Phi_n(x)=\begin{cases} \log \frac{1}{|x|}, \mbox{ for }n=2, \\
|x|^{n-2}, \mbox{ for }n\geq 3.\end{cases}.$$

The connection between $\mathsf{R}_{n-1}$ and removability can be understood as follows. Let $E$ be a closed subset of $\R^n$. If $E\subset \R^n$ is $(n-1)$--upper regular and $\mathsf{R}_{n-1}$ is bounded on $L^2(\mathcal{H}^{n-1}|_{E})$ (where $\mathcal{H}^{n-1}|_{E}$ is the $(n-1)$-dimensional Hausdorff measure restricted on $E$) then $E$ is \textbf{not} RLH, see \cite[Theorem 4.4]{MR1372240}. On the other hand, if $\mathcal{H}^{n-1}(E)<\infty$ and $E$ is not RLH, then there exists a Borel set $F \subset E$  with $\mathcal{H}^{n-1}(F)>0$ such that $R^{n-1}$ is bounded in $L^2(\mathcal{H}^{n-1}|_F)$, see \cite{volberg}.  Using deep methods from non-homogeneous harmonic analysis, David and Mattila \cite{dm} (in $\R^2$) and  Nazarov, Tolsa, and Volberg \cite{ntv, ntv2} (in $\R^n, n \geq 3$) characterized RLH sets as the purely \((n - 1)\)-unrectifiable sets in \(\mathbb{R}^n\) (the sets which intersect every \(C^1\) hypersurface in a set of vanishing \((n - 1)\)-dimensional Hausdorff measure).

Recently, significant efforts have been made towards the extension of classical Euclidean analysis
and geometry into general non-Riemannian spaces, including Carnot groups and more abstract
metric measure spaces. In particular, sub-Riemannian analogues of the Laplacian, known as sub-Laplacians, have been extensively studied in Carnot groups and sub-Riemannian manifolds, with foundational contributions from Stein, Folland, and others \cite{stfol,fol}. In particular, a seminal result of Folland \cite{fol} ensures the existence of a fundamental solution for every sub-Laplacian on Carnot groups with homogeneous dimension $Q \geq 3$. A comprehensive overview of the theory of sub-Laplacians can be found in \cite{BLU}. 

Given a sub-Laplacian $\Delta_\G$ in a Carnot group $\G$, we will call solutions to the equation $\Delta_\G f=0$, $\Delta_\G$-harmonic, or simply harmonic, functions. Harmonic functions on Carnot groups have been studied extensively, see e.g. \cite{BLU}, and naturally one can also study removable sets for Lipschitz $\Delta_\G$-harmonic functions ($\G$-RLH sets) in Carnot groups. Here and in the following, we consider Lipschitz functions $f: (A, \|\cdot\|)\to \R$ where $A \subset \G$ and $\|\cdot\|$ is a homogeneous norm in $\G$. It was proven in \cite{CM,CMT} that if $\G$ is a Carnot group of homogeneous dimension $Q$ then, similarly to the analogous problem in Euclidean spaces, the critical dimension for $\G$-RLH sets is $Q-1$.

As in the Euclidean case, singular integrals on sets of dimension $Q-1$ play a crucial role in the study of $\G$-RLH sets on Carnot groups with homogeneous dimension $Q$. To better explain this, fix a sub-Laplacian $\Delta_\G$ on a Carnot group $\G$ and denote by $\Gamma$ its fundamental solution. Recall that in $\R^n$, the relevant singular integral for removability is the \((n - 1)\)-dimensional Riesz transform $\mathsf{R}_{n-1}$, whose kernel
is the gradient of the fundamental solution of the Laplacian. When one studies $\G$-RLH sets (with respect to $\Delta_\G$), then has to deal with a new convolution-type singular integral whose kernel is
\begin{equation}
\label{rieszker}
\mathsf{R}(q)= \nabla_{\mathbb{G}} \Gamma(q), \quad q\in \mathbb{G}\backslash \{\zero\},
\end{equation}
the horizontal gradient of the fundamental solution of the sub-Laplacian $\Delta_\G$. We will call $\mathsf{R}$, the \textit{Carnot Riesz kernel} and its associated singular integral, the \textit{Carnot Riesz transform}. We record that $\mathsf{R}$ is a homogeneous $(Q-1)$-dimensional Calder\'on-Zygmund kernel, see Section \ref{sec:prelim} for more details.

In $\R^n$, the $(n-1)$-dimensional Riesz transform $\mathsf{R}_{n-1}$ is bounded in $L^2(\mathcal{H}^{n-1}|_{\Sigma})$ for any $(n-1)$-dimensional Lipschitz graph. This is a foundational result in harmonic analysis and geometric measure theory due to Coifman, McIntosth and Meyer \cite{cmm} (for $n=2$) and Coifman, David and Meyer \cite{cdm} (for $n>3$). Moreover, David proved in \cite{david88} that any convolution type SIO associated to an odd, smooth $1$-codimensional Calder\'on-Zygmund kernel is $L^2$ bounded on $1$-codimensional Lipschitz graphs. 

The sub-Riemannian analogue of this problem poses several new and fascinating challenges which do not exist in the Euclidean case. First of all, one has to find meaningful substitutes of Lipschitz graphs in Carnot groups. This is not straightforward as Carnot groups cannot be viewed as Cartesian products of subgroups. One could consider Lipschitz maps from Euclidean spaces, but this approach fails as well, see e.g. \cite{AK,mag}.
It thus becomes necessary to introduce notions of intrinsic graphs that are compatible with the Carnot group structure. Intrinsic Lipschitz graphs (ILGs) on Carnot groups were introduced by Franchi, Serapioni, and Serra Cassano in \cite{FSS1} and they have played pivotal role in advancing sub-Riemannian geometric measure theory, see e.g.  \cite{SCnotes} and \cite{perttirev} for some informative overviews. ILGs can be thought of as the Carnot group counterparts of $1$-codimensional  Lipschitz graphs in Euclidean spaces. Similarly to Euclidean Lipschitz graphs they can be defined via a cone condition; see Section \ref{sec:ilgs}. It is imporant to mention that similarly to Lipschitz graphs in $\R^n$, ILGs in Carnot groups of step 2 satisfy an analogue of Rade\-macher's theorem; their sub-Riemannian blow-ups resemble vertical hyperplanes almost everywhere \cite{FSSCDiff}. 

Another significant challenge arises from the fact that the cancellation properties of  Carnot Riesz kernels are significantly different from its Euclidean counterpart. While Euclidean Riesz kernels are antisymetric, Carnot Riesz kernels are \textit{dilation antisymmetric}. 
This means that if $\G$ is a Carnot group of step $s$; 
$$\mathsf{R}(\delta_{-1}(p))=\mathsf{R}(-p_1,p_2,-p_3, \dots, (-1)^s p_s)=-\mathsf{R}(p) \quad \mbox{ for } p=(p_1,\dots,p_s) \in \G \stm \{\zero\},$$
where we used exponential coordinates and $(\delta_t)_{t \in \R}$ denotes the one-parameter family of group dilations on $\G$,
see Sections \ref{sec:carnot} and \ref{sec:ker} for more details. We note that if $\G$ has step $2$ then \textit{dilation antisymmetry} agrees with the notion of \textit{horizontal antisymmetry} which appeared in \cite{CFO19}, and in that case: 
$$\mathsf{R}(-z,t)=-\mathsf{R}(z,t) \quad \mbox{ for } (z,t) \in \G \stm \{\zero\}.$$

 The different symmetries of $\mathsf{R}$ have surprising consequences. It was recently proved in \cite{CLYRiesz} that in the first Heisenberg group $\mathbb{H}^1$, there exist compactly supported ILGs $\Sigma$ where the Heisenberg Riesz transform is unbounded in $L^2(\mathcal{H}^3|_{\Sigma})$. This stands in sharp contrast with the Euclidean case and suggests that other novel phenomena should be expected in the study of SIOs on $1$-codimensional subsets of Carnot groups. 

In the case of the first Heisenberg group, some positive results are known if one assumes some extra regularity conditions on the ILGs.  In \cite{CFO19} it was shown that any horizontally antisymmetric, $3$-dimensional  Calder\'on-Zygmund kernel  defines an $L_2$--bounded singular integral on compactly supported intrinsic $C^{1,\alpha}$ graphs. In \cite{FORiesz}, F\"assler and Orponen proved that the Heisenberg Riesz transform is $L_2$--bounded on intrinsic Lipschitz graphs which satisfy some extra vertical regularity conditions. 

In this paper we study convolution type singular integral operators on $1$--codimensional ILGs of step--$2$ Carnot groups. Let $\G$ be a step--$2$ Carnot group of homogeneous dimension $Q$ equipped with a strongly homogeneous norm $\|\cdot\|$, see Section \ref{sec:carnot} for the definitions. We fix a $(Q-1)$-dimensional Calder\'on-Zygmund kernel $K: \mathbb{G}\backslash\{\zero\} \to \mathbb{R}^d, d \in \mathbb{N},$ and we consider truncated SIOs defined by
$$
T_{\mu,\ve} f(p) = \int_{\|q^{-1}\cdot p\|>\ve} K(q^{-1}\cdot p)f(p)\,d\mu(q), f\in L^2(\mu), \ve>0,
$$
where $\mu$ is a positive $(Q-1)$--upper regular measure. The singular integral operator  $T:=T_K$ (associated to $K$) is bounded on $L^2(\mu)$, if the operators
$$
f\mapsto T_{\mu, \epsilon}f
$$
are bounded on $L^2(\mu)$ with constants independent of $\epsilon>0$.

We prove that if $K$ satisfies a weak cancellation condition known as \textit{annular boundedness (AB)}, see Definition \ref{AB}, and $\mu$ is supported on a $C^{1,\alpha}$ intrinsic graph which decays at infinity, then $T_{K,\mu}$ is bounded in $L^2(\mu)$. We thus generalize earlier results from \cite{CFO19} which were valid for compactly supported $C^{1,\alpha}$ intrinsic graphs of the first Heisenberg group (equipped with the Koran\'yi norm).

Before stating our main theorem we record that a vertical hyperplane is a $1$-codimensional subgroup of $\mathbb{G}$ whose tangent space is transversal to a vector in the first layer of $\text{Lie}\,(\mathbb{G})$.

\begin{theorem}\label{theoreminfinity}
Let $\G$ be a step-2 Carnot group equipped with a strongly homogeneous norm $\|\cdot\|$ and let $\mathbb{W}$ be a vertical hyperplane. Let $\alpha>0$ such that $\phi\in C^{1, \alpha}(\mathbb{W})$ and let $\Sigma$ be the intrinsic graph of $\phi$. Further assume that there exist $0<\gamma, \theta \in (0,1)$ such that for every $w=(x,z)\in \mathbb{W}, p\in \Sigma$,
\begin{equation}\label{infinity}
|\phi^{(p^{-1})}(x,z)| \lesssim \|(x,z)\|^{1-\theta}, \qquad |\nabla^\phi \phi (x,z)|\lesssim \|(x,z)\|^{-\gamma}.
\end{equation}

 Let $K:\mathbb{G}\backslash\{\zero\} \to \mathbb{R}^d$ be a $(Q-1)$-dimensional Calder\'on-Zygmund kernel such that $K$ and $K^\ast$ (where $K^*(p):=K(p^{-1})$) satisfy the annular boundedness condition. Then, the associated singular integral is bounded in $L^2(\mu)$ for any $(Q-1)$-Ahlfors-David regular measure $\mu$ supported on the intrinsic graph of $\phi$.
\end{theorem}

We now discuss the assumptions of Theorem \ref{theoreminfinity} in more detail. First, the notion of a strongly homogeneous norm is more restrictive than the widely used notion of a homogeneous norm (a strongly homogeneous norm $\|\cdot\|$ satisfies $\|\delta_t(p)\|=|t| \|p\|$ for $t \in \R$, while a homogeneous norm only has to satisfy the previous equation for $t>0$). However, several homogeneous norms which appear frequently in the literature are  strongly homogeneous. For example, the Carnot-Caratheodory norm, the gauge norms of Laplacians, and the Koran\'yi norms are strongly homogeneous, see Section \ref{sec:carnot} for proofs of these claims.

Assumption \eqref{infinity} is a natural condition ensuring fast decay at infinity. Our condition is checkable and it is trivially satisfied by compactly supported graphs. 

The cancellation assumptions in Theorem \ref{theoreminfinity} are quite general and they are satisfied by all the interesting kernels that we are aware of. In certain cases, they can be simplified. For example, if  $\|\cdot\|$ is symmetric then we only need to assume that $K$ satisfies the AB condition. In particular, our theorem applies in the case when $\|\cdot\|$ is symmetric and $K$ is antisymmetric. More importantly, the cancellation conditions in Theorem \ref{theoreminfinity} are satisfied by horizontally antisymmetric kernels. As a corollary:

\begin{corollary}
\label{maincoro}
Let $\G$ be a step-2 Carnot group equipped with a strongly homogeneous norm $\|\cdot\|$ and let $\mathbb{W}$ be a vertical hyperplane. Let $\alpha>0$ such that $\phi\in C^{1, \alpha}(\mathbb{W})$ and furthermore assume that $\phi$ satisfies \eqref{infinity} for some $\gamma>0, \theta \in (0,1)$. Then, the Carnot Riesz transform, formally defined by
$$\mathcal{R}_{\mu} f (p)=\int \mathsf{R}(q^{-1}\cdot p) f(q) d \mu(q)$$
where $\mathsf{R}$ is the Carnot Riesz kernel from \eqref{rieszker}, is bounded in $L^2(\mu)$ for any $(Q-1)$-Ahlfors-David regular measure $\mu$ supported on the intrinsic graph of $\phi$.
\end{corollary}

Finally, our approach has applications in the study of removability for Lipschitz harmonic function in Carnot groups. First, we record that:
\begin{theorem}
\label{thm-removintro}
Let $\G$ be a Carnot group equipped with homogeneous norm $\|\cdot\|$. Assume that $\mu$ is a positive Radon measure on $\mathbb{G}$, satisfying the growth condition $\mu(B(p,r))\leq Cr^{Q-1}$ for $p\in \mathbb{G}$ and $r>0$, and such that the support $\text{supp}\,\mu$ has locally finite $(Q-1)$-dimensional Hausdorff measure. If the Carnot Riesz transform $\mathcal{R}_{\mu}$ is bounded on $L^2(\mu)$, then $\text{supp}\,\mu$ is not removable for Lipschitz harmonic functions.
\end{theorem}
Theorem \ref{thm-removintro} was proved in \cite{CFO19} for the case of the first Heisenberg group $\mathbb{H}^1$. The proof in the case of general Carnot groups is very similar and a new argument is required for only one step in the proof. We provide the details for the required modifications in Section \ref{sec:rem} as well as an outline of the proof from \cite{CFO19} for the convenience of the reader.

As a rather straightforward corollary of Theorems \ref{maincoro} and \ref{thm-removintro} we obtain that:
\begin{corollary}\label{riesznonremovable}
 Let $\G$ be a step-2 Carnot group equipped with a strongly homogeneous norm $\|\cdot\|$ and let $\mathbb{W}$ be a vertical hyperplane. Let $\alpha>0$ such that $\phi\in C^{1, \alpha}(\mathbb{W})$ and furthermore assume that $\phi$ satisfies \eqref{infinity} for some $\gamma>0, \theta \in (0,1)$. If $E$ is a closed subset of the intrinsic graph of $\phi$ with positive $(Q-1)$-dimensional Hausdorff measure, then $E$ is not removable for Lipschitz harmonic functions.
\end{corollary}

Moreover, arguing as in \cite[Corollary 5.7]{CFO19} we obtain:
\begin{corollary}
\label{cororem2}
Let $\G\cong (\R^N,\cdot)$ be a step-2 Carnot group of homogeneous dimension $Q$ and let $\Omega \subset \R^{N-1} \cong \mathbb{W}$ be an open set. If $\phi$ is Euclidean $C^{1,\alpha}$ on $\Omega$ and $E$ is a closed subset of the intrinsic graph of $\phi$ over $\mathbb{W}$ with positive $(Q-1)$-dimensional Hausdorff measure, then $E$ is not removable for Lipschitz harmonic functions.
\end{corollary}

The proof of Theorem \ref{theoreminfinity} follows the same scheme as in \cite{CFO19}. The key step, which we consider that is of independent interest, provides quantitative affine approximation for intrinsic $C^{1,\alpha}$ graphs in any step-2 Carnot group, see Proposition \ref{holder}. This was previously proven only for Heisenberg groups (see \cite[Proposition 4.1]{CFO19} for $\mathbb{H}^1$ and \cite[Proposition 3.29]{DDFO22} for Heisenberg groups of arbitrary dimension) and our proof naturally differs quite significantly. 

Finally, we note that the reason we focus only on step-$2$ Carnot groups  is that the proof of Proposition \ref{holder} for higher-step groups would require commutators of higher degree, and $C^{1, \alpha}$ regularity would not be sufficient. In addition, it is unknown if a Rademacher-type theorem for intrinsic $1$-codimensional Lipschitz graphs (asserting that they can be approximated by vertical hyperplanes at almost every point as in \cite{FSSCDiff}) holds in Carnot groups of step $3$ or higher.

\begin{roadmap}
In Section~\ref{sec:prelim}, we introduce all the necessary background related to Carnot groups, singular integrals, and intrinsic  Lipschitz and  $C^{1,\alpha}$ graphs in Carnot groups of step $2$. In Section~\ref{curves} we prove  an effective H\"older estimate for the (vertical) affine approximation of intrinsic $C^{1,\alpha}$ graphs.  In Section \ref{4} we provide the proof of the Theorem \ref{theoreminfinity}. Finally, in Section \ref{sec:rem} we discuss removability for Lipschitz harmonic functions. In particular we prove Theorem \ref{thm-removintro}.
\end{roadmap}

\textbf{Acknowledgement.} We thank the anonymous referee carefully reading our manuscript and for providing useful comments.

\section{Preliminaries}\label{sec:prelim}
\subsection{Carnot groups} \label{sec:carnot} A \textit{Carnot} group is a connected and simply connected nilpotent Lie group whose associated Lie algebra $\mathfrak{g}$ admits a \emph{stratification} of the form
$$
\mathfrak{g} = \fv_1 \oplus \cdots \oplus \fv_s, 
\quad [\fv_1,\fv_i] = \fv_{i+1}  \text{ for } i=1,\dots,s-1,
\quad [\fv_1,\fv_s] =\{0\},
$$
where $\fv_1,\dots,\fv_s$ are non-zero subspaces of  $\mathfrak{g}$. The integer $s \geq 1$ is called the \textit{step} of $\mathbb{G}$. Notably, the  Lie algebra $\mathfrak{g}$ is generated by iterated Lie brackets of elements from the first layer $\fv_1$ of the stratification. This layer, referred to as the \textit{horizontal layer} of the Lie algebra, consists of elements known as \textit{horizontal tangent vectors}. We let $n_1 = \dim(\fv_1)$.

Since $\mathbb{G}$ is a connected, simply connected and nilpotent Lie group, the
\textit{exponential map} $\exp:\fg\to\G$ is a global diffeomorphism. Therefore, we can naturally identify $\mathbb{G}$ with a Euclidean space $\G = \R^N$. We call elements of $\exp(\fv_1)$ {\em horizontal vectors}.

Let $\mathbb{K} = \exp(\fv_2 \oplus ... \oplus \fv_s)$. This is the commutator subgroup of $\G$ and so $\G / \mathbb{K} \cong \R^{n_1}$.  Let $\pi : \G \to \R^{n_1}$ be the natural projection map. We say a subgroup of $\G$ is {\em vertical} if it contains $\mathbb{K}$. These are also the subgroups of the form $\pi^{-1}(Y)$ where $Y$ is a subspace of $\R^{n_1}$. We say a set $V \subset G$ is a {\em vertical hyperplane} if $V = \pi^{-1}(X)$ where $X \subset \R^{n_1}$ is a $n_1-1$-dimensional affine plane. 

The exact formula for the group law in $\mathbb{G}$
follows from the {\it Baker--Campbell--Hausdorff (BCH) formula}, see e.g. \cite[Theorem 2.2.13]{BLU}, which asserts that
\begin{equation}\label{BCHformula}
\exp(U) * \exp(V) = \exp( U + V + \tfrac12 [U,V] + \tfrac1{12} ([U,[U,V]] + [V,[V,U]]) + \cdots), \quad \mbox{ for }U,V \in \fg.
\end{equation}
We remark that since the Lie group $\mathbb{G}$ is nilpotent, the sum on the right hand side of \eqref{BCHformula} contains finitely many terms.  Using exponential coordinates, $p = (U_1,\dots, U_s)$ for $p = \exp(U_1 + \dots+ U_s)$  with $U_i \in \fv_i, i=1,\dots, s,$ we see that $\zero= (0,\dots, 0)$ and by \eqref{BCHformula} that $p^{-1}=(-p_1,\dots)$, see also \cite[p.36]{ruzhansky}.

For $t \in \R
$, we define $\delta_t:\fg\to\fg$ by setting $\delta_t(X)=t^iX$ if $X\in \fv_i$ and extending the map to $\fg$ by linearity. Using the canonical identification of the Lie algebra $\fg$ with the group $\mathbb{G}$ (via the exponential mapping) we can define dilations on $\mathbb{G}$, which we'll also denote by $\delta_t$.

A function $\|\cdot\|: \G \to [0,\infty)$ is called a \textit{strongly homogeneous norm} if it is continuous with respect to the Euclidean metric and satisfies
\begin{enumerate}
\item \label{posdil}$\|\delta_t(p)\| =|t| \|p\|$ for all $t \in \R$ and $p \in \G$,
\item \label{norm0} $\|p\|=0$ if and only if $p=0$.
\end{enumerate}
A continuous function $\|\cdot\|: \G \to [0,\infty)$ which satisfies \eqref{norm0} and \eqref{posdil} for positive $t$ is called  {\em homogeneous}. A homogeneous norm is called \textit{symmetric} if $\|p^{-1}\|=\|p\|$ for all $p \in \G$.

It is well known, see e.g. \cite[5.1.4]{BLU}, that homogeneous norms are globally equivalent in $\G$, meaning that if $\|\cdot\|_1,\|\cdot\|_2$ are homogeneous norms in $\G$ then there exists a constant $c>1$ such that
$$c^{-1} \,\|p\|_2 \leq \|p\|_1 \leq c \,\|p\|_2 \mbox{ for all }p \in \G.$$ If $\|\cdot\|$ is homogeneous then $d(p,q)=\|q^{-1} \cdot p\|$ satisfies a relaxed triangle inequality, see \cite[Proposition 5.1.8]{BLU}. A {\it strongly homogeneous} metric $d$ on $\G$ is any metric which is continuous (with respect to the Euclidean topology), left invariant
and $1$-homogeneous with respect to the dilations
$(\delta_t)_{t \in \R}$; i.e. 
$$
d(\delta_t(p),\delta_t(q)) = |t|\, d(p,q)
$$
for $p,q\in\G$ and $t \in \R$. If $d$ is a strongly homogeneous metric on $\G$, then the function $p \to d(p,0)$ is a strongly homogeneous norm.

The Hausdorff
dimension of $\G$ with respect to any strongly homogeneous metric is equal to
\begin{equation}\label{Q}
Q = \sum_{i=1}^s i \, \dim \fv_i,
\end{equation}
which is also frequently referred to as the \textit{homogeneous dimension of $\G$.} In this paper we will always assume that $Q \geq 3$.

We remark that if $\G$ is not Abelian (i.e., when $s>1$), then $Q>N$ where $N$ denotes the topological
dimension of $\G$. As a consequence, strongly homogeneous metrics are not bi-Lipschitz equivalent to Riemannian metrics on
the Euclidean space $\R^N$. It is not difficult to check that the Jacobian determinant of the dilation
$\delta_t$ is everywhere equal to $t^Q$. Using this fact, one can show that if $d$ is a homogeneous metric and $B_d(p,r)=\{q \in \G: d(p,q)<r\}$ then the Haar measure of $B_d(p,r)$ equals $c_d r^Q$ where $c_d$ is the Haar measure of $B_d(0,1)$.

In this paper, we will be mainly concerned with step-2 Carnot groups. In this case, \eqref{BCHformula} becomes
\begin{equation}\label{BCHformula2}
\exp(U) * \exp(V) = \exp( U + V + \tfrac12 [U,V] ).
\end{equation}
Using exponential coordinates $p = (U_1,U_2)$ for $p = \exp(U_1 + U_2)$ with $U_1 \in \fv_1$ and $U_2 \in \fv_2$, and an application of \eqref{BCHformula2} gives
\begin{equation}\label{step2grouplaw}
p \cdot q = (U_1+V_1,U_2+V_2+\tfrac12[U_1,V_1])
\end{equation}
for $p=(U_1,U_2)$ and $q=(V_1,V_2)$. From now on we fix a Carnot group $\mathbb{G}$ of step $2$ and we choose a basis $\{X_1,\ldots,X_m\}$ for the first layer $\fv_1$. This basis is sometimes referred to as the \textit{horizontal frame}, which linearly spans all the horizontal directions. We also choose a basis $\{Z_1, \ldots, Z_{(Q-m)/2}\}$ for the second layer of $\text{Lie}\,(\mathbb{G})$. Since every point  of $\mathbb{G}$ can be written as
$$
\exp\left(\sum_{i=1}^m x_i X_i + \sum_{i=1}^{\frac{Q-m}{2}} z_i Z_i \right),
$$ we will denote points in $\G$ by $p=(x,z)$ where $x \in \R^m, z \in \R^{\frac{Q-m}{2}}$. Thus, we can view $\G$ as $(\R^N, \cdot)$ where $N:=\frac{Q+m}{2}$. 

The horizontal gradient of a function $u:\Omega \to \mathbb{R}$ on an open set $\Omega \subseteq \mathbb{G}$ is defined by
$$
\nabla_{\mathbb{G}} u = \sum_{j=1}^m (X_j u)X_j,
$$
and the sub-Laplacian of $u$ is
$$
\Delta_{\mathbb{G}} u = \sum_{j=1}^m X_j^2 u.
$$
By a celebrated result of Folland \cite{fol}, the sub-Laplacian $\Delta_{\mathbb{G}}$ admits a
fundamental solution which we will denote by $\Gamma$. We refer the reader to \cite[Chapter 5]{BLU} for several important properties of $\Gamma$. 

We will focus on pairs $(\G, \|\cdot\|)$ where $\|\cdot\|$ is a strongly homogeneous norm. Below we list some natural such  examples:
\begin{itemize}[wide, labelwidth=!, labelindent=0pt]
\item $(\G, \|\cdot\|_{cc})$ where $\|\cdot\|_{cc}$ denotes the \textit{Carnot-Carath\'eodory norm} on $\G$:
\begin{align*}
\|z\|_{cc}& = \inf\big\{r>0\big| \exists\text{ an absolutely continuous curve } \gamma:[0,1]\to \mathbb{G} \text{ such that
}\gamma(0)=0, \gamma(1)=z, \\
&\qquad \qquad  \gamma'(s) = \sum_{j=1}^m ra_j(s)X_j(\gamma(s)) \text{ where the } a_j \text{ are continuous and satisfy} \sum_{j=1}^m |a_j(s)|^2<1\big\}.
\end{align*}
We note that $d_{cc}(p,q)=\|q^{-1} \cdot p\|_{cc}$ defines a metric on $\G$, known as the Carnot-Carath\'eodory metric. It is well known that the Carnot-Caratheodory norm is homogeneous (see e.g. \cite[Section 5.2 and Chapter 19]{BLU}). Below we provide a proof of its strong homogeneity since we could not find a proof in the literature. 
\begin{lemma}
The Carnot-Caratheodory norm $\|\cdot\|_{cc}$ is strongly homogeneous.
\end{lemma}
\begin{proof}
We only verify that $\|\delta_t(z)\|_{cc} = |t|\|z\|_{cc}$ for each $t\in \mathbb{R}$; the other properties in the definition of the homogeneous metric are well known. Fix arbitrary $\epsilon>0$ and fix an arbitrary absolutely continuous curve $\gamma: [0,1] \to \mathbb{G}$ such that $\gamma(0)=0, \gamma(1)=z$ and such that there exist continuous functions $a_1(s), \ldots, a_m(s)$ satisfying
$$
\gamma'(s)= \sum_{j=1}^m (\|z\|_{cc}+\epsilon) a_j(s)X_j(\gamma(s)), \quad \sum_{j=1}^m |a_j(s)|^2<1.
$$
Consider $\tilde \gamma: [0,1]\to \mathbb{G}$ with $\tilde \gamma(s):= \delta_t(\gamma(s))$, which is a curve that goes from 0 to $\delta_t(z)$.
We have
$$
\tilde \gamma'(s) = (\delta_t)_* \gamma'(s)=\sum_{j=1}^m (\|z\|_{cc}+\epsilon) a_j(s)(\delta_t)_*X_j(\gamma(s)) =\sum_{j=1}^m t(\|z\|_{cc}+\epsilon) a_j(s)X_j(\tilde \gamma(s)).
$$
By the arbitrariness of $\epsilon>0$, we have $\|\delta_t(z)\|_{cc} \leq |t|\|z\|_{cc}$.

As $\gamma = \delta_{1/t}(\tilde{\gamma})$, we can repeat the argument above with the roles of $\gamma$ and $\tilde{\gamma}$ reversed to get the reverse inequality. This prooves the lemma.
\end{proof}
\item $(\G,\|\cdot\|_{\Gamma})$ where $\|\cdot\|_{\Gamma}$ is defined as
$$
\|x\|_{\Gamma}:=\begin{cases} 
\big(\Gamma(x)\big)^{\frac{1}{2-Q}}, \quad &x\neq 0,\\
0, \quad &x=0,
\end{cases}
$$
and $\Gamma$ is the fundamental solution of the sub-Laplacian $\Delta_{\G}$.
\begin{lemma}
$\|\cdot\|_{\Gamma}$
is a strongly homogeneous norm.
\end{lemma}

\begin{proof}
It is well known that $\|\cdot\|_{\Gamma}$ is a homogeneous norm, see e.g. \cite[Proposition 5.4.2]{BLU}. The strong homogeneity of $\|\cdot\|_{\Gamma}$ follows by
\begin{equation}
\label{gaugestronghomo}
\Gamma(\delta_\lambda(x)) =|\lambda|^{2-Q} \Gamma(x), \quad \forall \lambda \neq 0,
\end{equation}
which is essentially proved in \cite[Proposition 5.3.12]{BLU}. Although in \cite[Proposition 5.3.12]{BLU}, \eqref{gaugestronghomo} is only proved for $\lambda>0$, the argument works also for $\lambda<0$ by using
$$\Gamma'(x)= |\lambda|^{Q-2} \Gamma(\delta_\lambda(x))$$
in the place of $\Gamma'(x)= \lambda^{Q-2} \Gamma(\delta_\lambda(x))$ and by noting that the absolute value of the Jacobian of $\delta_\lambda$ is $|\lambda|^Q$.
\end{proof}

\item $(\G, \|\cdot\|_{K})$ where 
$\|p\|_{K}=\sqrt[4]{|x|^4+|z|^{2}}$ and
$p=(x,z), x \in \R^m, z \in \R^{\frac{Q-m}{2}}$. Clearly, $\|\cdot\|_{K}$ is a strongly homogeneous norm.

\end{itemize}


\subsection{Kernels and singular integrals on Carnot groups.} \label{sec:ker} Throughout this section we will assume that $\mathbb{G}$ is a Carnot group equipped with a strongly homogeneous norm $\|\cdot\|$. We also let $d(p,q):=\|q^{-1} \cdot p\|$ and $B(p,r)=\{q \in \G: d(p,q)<r\}$. In several instances we will restrict our attention to step-2 groups.
\begin{definition}
\label{czkernel}
Let $K: \mathbb{G}\backslash \{0\} \to \mathbb{R}^d$ be a continuous function. We say that $K$ is a $(Q-1)$-dimensional Calder\'on-Zygmund (CZ) kernel in $\mathbb{G}$ if there exist $\kappa \in (0,1),\beta \in (0,1], C_K \geq 1,$ such that:
\begin{itemize}
    \item (Growth condition) 
    \begin{equation}
    \label{czgrowth}
|K(p)| \leq C_K \frac{1}{\|p\|^{Q-1}}, \quad \forall p\in \mathbb{G}\backslash \{0\},
    \end{equation}
    \item (H\"older continuity)
    \begin{equation}
    \label{czhol}
    |K(p_1)-K(p_2)| \leq C_K \frac{\|p_2^{-1} \cdot p_1\|^\beta}{\|p_1\|^{Q-1+\beta}}, \quad \forall p_1, p_2\in \mathbb{G}\backslash \{0\} \text{ with } d(p_1, p_2) \leq \kappa \|p_1\|.
    \end{equation}
\end{itemize}
\end{definition}

\begin{remark}
    Note that if $K$ is a CZ kernel for one homogeneous norm, it is a CZ kernel for any homogeneous norm. However, the value of $\kappa$ and the implicit constants can change depending on the choice of norm.
\end{remark}

Recall that a function $f: \G  \stm \{\zero\} \to \R^d$ is called \textit{$\lambda$-homogeneous}, for $\lambda \in \R$, if
$$f(\delta_t p)=t^{\lambda} f(p).$$

The following proposition provides us with a large class of $(Q-1)$-dimensional CZ kernels.
\begin{proposition}
\label{prop:homkernels}
If $K: \G  \stm \{\zero\} \to \R^d$ is a $(1-Q)$-homogeneous $C^1$ function, then it is a $(Q-1)$-dimensional CZ kernel. 
\end{proposition}
\begin{proof} To verify \eqref{czgrowth} let $p \in \G \stm \{0\}$ and note that
$$K(\delta_{\|p\|}(\delta_{1/\|p\|}(p)))=\|p\|^{1-Q}K(\delta_{1/\|p\|}(p)).$$
Now, \eqref{czgrowth}  follows because by the continuity of $K$ and $\|\cdot\|$,
$$\sup_{q \in \G: \|q\|=1} |K(q)|<\infty.$$

For \eqref{czhol} we first record that by \cite[Proposition 1.7]{stfol} there exists some constant $A>0$ such that
\begin{equation}
\label{homeq1}
|K(p\cdot q)-K(p)|\leq A \|p\| \|q\|^{-Q}
\end{equation}
for $\|q\| \leq \|p\|/2$. By \cite[Corollary 5.1.5]{BLU} there exists some $B\geq 1$ such that
$$B^{-1}\|p \|\leq \|p^{-1}\| \leq B\|p\|\quad\mbox{ for all }p \in \G.$$
Let $z_1, z_2 \in \G$ such that 
\begin{equation}
\label{holquot}
\|z_2^{-1} \cdot z_1\| \leq \frac{1}{2B} \|z_1\|.
\end{equation}
Then, for such $z_1,z_2$,
$$\|z_1^{-1}\cdot z_2\| \leq  B\|z_2^{-1} \cdot z_1\| \leq \frac{\|z_1\|}{2}.$$
Therefore, if $z_1,z_2$ satisfy \eqref{holquot} then
$$|K(z_2)-K(z_1)|=|K(z_1 \cdot z_1^{-1} \cdot z_2)-K(z_1)| \overset{\eqref{homeq1}}{\leq}A \frac{\|z_1^{-1} \cdot z_2\|}{\|z_1\|^Q} \leq A B \frac{\|z_2^{-1} \cdot z_1\|}{\|z_1\|^Q}.$$
Thus, $K$ satisfies \eqref{czhol} with $\beta=1$ and $\theta=\frac{1}{2B}$.
\end{proof}
Let $\mu$ be positive Radon measure  in $\G$. The measure $\mu$ is called \textit{$(Q-1)$-Ahlfors-David regular} (in short, $(Q-1)$-ADR) if
\begin{equation}
\label{eq:ADR}
C^{-1}r^{Q-1} \leq \mu(B(p,r)) \leq Cr^{Q-1}, \quad \forall p\in \supp \mu, 0<r \leq \diam(\supp \mu),
\end{equation}
for some constant $C\geq 1$. If only the right-hand side inequality in \eqref{eq:ADR} holds, we will say that $\mu$ is $(Q-1)$-upper regular.

Fix a $(Q-1)$-dimensional CZ kernel $K$ and a positive $(Q-1)$--upper regular measure. For $f\in L^2(\mu)$ and $\epsilon>0$, we define
$$
T_{\mu, \epsilon}f(p) = \int_{\|q^{-1}\cdot p\| >\epsilon} K(q^{-1}\cdot p)f(q)\,d\mu(q), \quad p\in \mathbb{G},
$$
where the expression on the right makes sense, due to the growth conditions of $K$ and $\mu$ and the Cauchy-Schwartz inequality. If moreover $\mu$ is finite, we let
$$T_{\epsilon}\mu(p) = \int_{\|q^{-1}\cdot p\| >\epsilon} K(q^{-1}\cdot p)\,d\mu(q), \quad p\in \mathbb{G}.
$$
\begin{definition}\label{SIO}
Given a $(Q-1)$-dimensional CZ kernel $K$ and a positive $(Q-1)$--upper regular measure, we say that the singular integral operator (SIO) $T$ associated to $K$ is bounded on $L^2(\mu)$, if the operators
$$
f\mapsto T_{\mu, \epsilon}f
$$
are bounded on $L^2(\mu)$ with constants independent of $\epsilon>0$.
\end{definition}

\begin{remark}
\label{adjointremark} Let $\mu$ be a positive $(Q-1)$--upper regular measure. If $K$ is a $(Q-1)$-dimensional CZ kernel, then the kernel $K^{\ast}$, defined by $K^{\ast}(p) := K(p^{-1})$, is the kernel of the formal adjoint $T_{\mu,\epsilon}^{\ast}$ of $T_{\mu,\epsilon}$, see \cite[Remark 2.3]{CFO19} for the details. It follows by \cite[Corollary 5.1.5]{BLU} that $K^{\ast}$ satisfies the growth condition \eqref{czgrowth} (with a different constant if $\|\cdot \|$ is not symmetric). Moreover, Lemma \ref{l:standardHolder}  implies that $K^{\ast}$ satisfies the H\"older condition \eqref{czhol}  with exponent $\beta/2$. We also record that if $f,g \in L^2(\mu)$:
\begin{align}
\label{eq:adjoint}\int (T_{\mu,\epsilon}f) g \, d\mu = \int (T^{\ast}_{\mu,\epsilon}g)f \, d\mu. \end{align}
\end{remark}

The defining conditions of $(Q-1)$-dimensional CZ kernels are not strong enough to guarantee $L^2(\mu)$-boundedness for the corresponding SIO, even when the measure $\mu$ is flat (i.e. when $\mu$ is the $(Q-1)$-dimensional Hausdorff measure restricted on a vertical hyperplane $\mathbb{W}$). Imposing the following mild cancellation condition will be sufficient for our purposes. 

For the following definition we will denote by $\mathcal{L}^k$ the $k$-dimensional Lebesgue measure. We also let
$$
\psi^r(p):= \psi (\delta_r(p))
$$
for $p \in \G$ and $r>0$.
\begin{definition}\label{AB} Let $\G$ be Carnot group. A continuous function $K: \mathbb{G}\backslash \{0\} \to \mathbb{R}$ satisfies the \textit{annular boundedness condition} (AB), if for every $\|\cdot \|$-radial $C^\infty$ function $\psi: \mathbb{G} \to \mathbb{R}$ satisfying $\chi_{B(0,1/2)} \leq \psi \leq \chi_{B(0,2)}$, there exists a constant $A_\psi \geq 1$ such that
$$
\Big| \int_{\mathbb{W}} \big(\psi^R(w)-\psi^r(w)\big)K(w)\,d\mathcal{L}^{N-1}(w)\Big| \leq A_\psi,
$$
for all $0<r<R<\infty$, and for all vertical hyperplanes $\mathbb{W}$.
\end{definition}
Annular boundedness is a broadly encompassing cancellation condition, which includes \textit{antisymmetric} and \textit{dilation antisymmetric} kernels. A function $K: \G \stm \{\zero\} \to \R^d$ is called antisymmetric if 
$$K(p^{-1})=-K(p),\quad \mbox{ for all }p \in \G \stm \{\zero\},$$
and  is called dilation antisymmetric if
$$K(\delta_{-1}(p))=-K(p),\quad \mbox{ for all }p \in \G \stm \{\zero\}.$$ Note that in Carnot groups of step $2$ the notion of dilation antisymmetric kernels coincides with horizontal antisymmetry as in \cite{CFO19} (or $\mathbb{H}$-oddness as in \cite{CLYRiesz}).
\begin{lemma}\label{antisymmetry}
Let $\G$ be a Carnot group and let $K:(\G,\|\cdot\|) \stm \{\zero\} \to \R^d$ be a continuous function.
\begin{enumerate}
\item \label{eq:khorantAB} If $K$
 is dilation
 antisymmetric then it satisfies the AB condition.
\item \label{eq:kantAB}If $\|\cdot \|$ is symmetric and $K$ is antisymmetric then $K$ satisfies the AB condition.
\end{enumerate}
\end{lemma}
\begin{proof}
We will only prove \eqref{eq:khorantAB}; the  proof of \eqref{eq:kantAB} is similar. Let $K:\G \stm \{\zero\} \to \R^d$ be a continuous dilation antisymmetric function. Since $\|\cdot\|$ is strongly homogeneous, $\|\delta_{-1}(p)\|=\|p\|$.  Fix a vertical hyperplane $\mathbb{W}$
and a radial function $\psi$ as in Definition \ref{AB}. Then, since $\delta_{-1}$ preserves Lebesgue measure on $\mathbb{W}$:
\begin{align*}
&\int_{\mathbb{W}} (\psi^R(p)-\psi^r(p))K(p)\,d\mathcal{L}^{N-1}(p)\\
&\quad\quad=\frac{1}{2}\int_{\mathbb{W}} \Big((\psi^R(p)-\psi^r(p))K(p) +(\psi^R(\delta_{-1}(p))-\psi^r(\delta_{-1}(p)))K(\delta_{-1}(p))\Big)\,d\mathcal{L}^{N-1}(p)\\
&\quad\quad=\frac{1}{2}\int_{\mathbb{W}} \Big((\psi^R(p)-\psi^r(p))K(p) +(\psi^R(p)-\psi^r(p))K(\delta_{-1}(p))\Big)\,d\mathcal{L}^{N-1}(p)=0.
\end{align*}
\end{proof}
We will now discuss two families of $(Q-1)$-dimensional CZ kernels where our theorems apply. 

\begin{remark}
\label{rem:ashasadj} Let $\G$ be Carnot group and $K$ be a $(Q-1)$-dimensional CZ kernel. If $K$ is antisymmetric or horizontally antisymmetric then $K^\ast$ (recall Remark \ref{adjointremark}) is, respectively antisymmetric or dilation antisymmetric.

Assume, for example, that $K$ is horizontally antisymmetric. As $\delta_{-1}$ is a homomorphism, we have that $\delta_{-1}(p^{-1}) \delta_{-1}(p) = \delta_{-1}(0) = 0$ so $\delta_{-1}(p)^{-1} = \delta_{-1}(p^{-1})$. Then
$$K^\ast (\delta_{-1}(p))=K(\delta_{-1}(p^{-1}))=-K(p^{-1})=-K^\ast(p).$$
The claim follows similarly for the case when $K$ is antisymmetric.
\end{remark}
\begin{example}\label{griesz}[The $(Q-1)$-dimensional $\G$-Riesz kernels] Let $\G$ be a Carnot group of homogeneous dimension $Q$. 
The \emph{$(Q-1)$-dimensional $\G$-Riesz kernel}. $\mathsf{R}:\G \stm \{\zero\} \to \R^m$ is defined as
\begin{equation}\label{rieszKernel} \mathsf{R}(p) = \nabla_{\G} \Gamma (p), \qquad p \in \G \stm \{\zero\}, \end{equation}
where $\Gamma$ is the fundamental solution of the sub-Laplacian. 

\begin{lemma}
\label{lem:rieszarecz}
The $(Q-1)$-dimensional $\G$-Riesz kernel is a $(Q-1)$-dimensional Calder\'on-Zygmund kernel.
\end{lemma} 
\begin{proof}
By \cite[Proposition 5.3.12]{BLU} we know that $\Gamma$ is $(2-Q)$-homogeneous.  Arguing exactly as in \cite[Lemma 4.1]{CLYRiesz} we derive that $\mathsf{R}$ is $(1-Q)$-homogeneous. Consequently, Proposition \ref{prop:homkernels} implies that $\mathsf{R}$ is a $(Q-1)$-dimensional Calder\'on-Zygmund kernel. \end{proof}


We will now show that $\mathsf{R}$ is a dilation antisymmetric kernel. To this end, we need the following simple lemma. 

\begin{lemma} For all $p \in \G \stm \{0\}$,
$$\Gamma(\delta_{-1}(p))=\Gamma(p).$$
\end{lemma}

\begin{proof}
By the definition of the fundamental solution $\Gamma$, \cite[Definition 5.3.1]{BLU}, we deduce that $\widetilde \Gamma(p)=\Gamma(\delta_{-1}(p))\in C^\infty(\mathbb{R}^N\backslash\{0\}) \cap L^1_{\text{loc}}(\mathbb{R}^N)$. We have $\widetilde \Gamma(p)\to 0$ as $p\to \infty$. Moreover, for every test function $\phi\in C_c^\infty(\mathbb{R}^N)$, by a change of variables,
\begin{equation*}
\begin{split}
\int_{\mathbb{R}^N} \widetilde \Gamma(p) \Delta_G\phi(p)\,dp&=\int_{\mathbb{R}^N} \Gamma(\delta_{-1}(p)) \Delta_\G\phi(p)\,dp\\
&=\int_{\mathbb{R}^N}\Gamma(p) (\Delta_\G\phi)(\delta_{-1}(p))\,dp\\
&= (-1)^2\int_{\mathbb{R}^N} \Gamma(p) \Delta_\G\Big(\phi(\delta_{-1}(p))\Big)\,dp\\
&=-\phi(\delta_{-1}(0))\\
&=-\phi(0).
\end{split}
\end{equation*}
Therefore, $\widetilde \Gamma$ is a fundamental solution of $\nabla_\mathbb{G}$. By the uniqueness of the fundamental solution, see \cite[Proposition 5.3.10]{BLU}, $\widetilde \Gamma =\Gamma$. 
\end{proof}

\begin{corollary}
The \emph{$(Q-1)$-dimensional $\G$-Riesz kernel} is dilation antisymmetric.
\end{corollary}

\begin{proof}
It suffices to show $-(\nabla_\G \Gamma)(\delta_{-1}(p))=\nabla_\G \Gamma(p)$. For any fixed horizontal vector field $W$, we have
\begin{align*}
W\Gamma(p) &= \lim_{h\to 0} \frac{\Gamma\big(p\cdot e^{hW}\big) -\Gamma(p)}{h}\\
& = -\lim_{h\to 0} \frac{\Gamma\big(\delta_{-1}(p)\cdot e^{-hW}\big) -\Gamma(\delta_{-1}(p))}{-h}\\
&= -(W\Gamma)(\delta_{-1}(p)).
\end{align*}
Therefore, $-(\nabla_\G \Gamma)(\delta_{-1}(p))=\nabla_\G \Gamma(p)$.
\end{proof}

\end{example}
\begin{example} [The $(Q-1)$-dimensional pseudo-Riesz kernels in $\G$] 
Let $\G$ be a Carnot group of step $s$. Let $\mathsf{P}:\G \stm \{\zero\} \to \R^{N}$ defined by
$$\mathsf{P}(p)=\left(\frac{p_1}{\|p\|^{Q}}, \frac{p_2}{\|p\|^{Q+1}}, \dots, \frac{p_s}{\|p\|^{Q+s-1}} \right), \quad \mbox{ for all }p=(p_1,\dots, p_s) \in \G \stm \{0\}.$$
This kernel resembles in form the one codimensional Riesz kernel in $\R^n$, i.e. the kernel $x|x|^{-n}$. Clearly, the kernel $\mathsf{P}$ is $(1-Q)$-homogeneous and Proposition \ref{prop:homkernels} implies that $\mathsf{P}$ is a $(Q-1)$-dimensional Calder\'on-Zygmund kernel. Moreover, if the norm $\|\cdot\|$ is symmetric the kernel $\mathsf{P}$ is antisymmetric and Proposition \ref{antisymmetry} \eqref{eq:kantAB} implies that $\mathsf{P}$ satisfies the $AB$ condition.
\end{example}

The following lemma will be applied in Section \ref{4} and it concerns Carnot groups of step $2$.
\begin{lemma}\label{l:standardHolder} Let $\G$ be Carnot group of step $2$. Assume $K \colon \mathbb{G} \setminus \{\mathbf{0}\} \to \R^{d}$ is a $(Q-1)$-dimensional CZ kernel which satisfies \eqref{czhol} for some $\beta$ and $\kappa$. Then, there exists some $C \geq 1$ depending on $\kappa$ and $\beta$ so that
\begin{equation}\label{form42} |K(q^{-1} \cdot p_{1}) - K(q^{-1} \cdot p_{2})| + |K(p_{1}^{-1} \cdot q) - K(p_{2}^{-1} \cdot q)| \lesssim \frac{\|p_{2}^{-1} \cdot p_{1}\|^{\beta/2}}{\|q^{-1}\cdot p_{1}\|^{Q-1 + \beta/2}} \end{equation}
for $q\in \mathbb{G}$, $p_{1},p_{2} \in \mathbb{G} \setminus\{q\}$ with $d(p_{1},p_{2}) \leq d(p_{1},q)/C$.
\end{lemma}

\begin{proof} Write $w_{1} := q^{-1} \cdot p_{1}$ and $w_{2} := q^{-1} \cdot p_{2}$. Then by left-invariance of $d$ and taking $C \geq 1/\kappa$, we get
$$d(w_{1},w_{2}) \leq \|w_1\|/C \leq \kappa\|w_{1}\|,$$
so the first summand in \eqref{form42} has the correct bound by the H\"older continuity in Definition \ref{czkernel}, even with $\beta/2$ replaced by $\beta$. Hence, to find a bound for the second summand, we only need to prove that
\begin{displaymath} |K(w_{1}^{-1}) - K(w_{2}^{-1})| \lesssim \frac{\|w_{2}^{-1} \cdot w_{1}\|^{\beta/2}}{\|w_{1}\|^{Q-1 + \beta/2}}. \end{displaymath}
We would like to apply the H\"older continuity in Definition \ref{czkernel} as follows,
\begin{equation}\label{form43} |K(w_{1}^{-1}) - K(w_{2}^{-1})| \lesssim \frac{\|w_{2} \cdot w_{1}^{-1}\|^{\beta}}{\|w_{1}\|^{Q-1 + \beta}}, \end{equation}
but we first need to make sure that $d(w_{1}^{-1},w_{2}^{-1}) \leq \kappa \|w_{1}\|$. Write $w_{1} = (x_{1},z_{1})$ and $w_{2} = (x_{2},z_{2})$. By the Baker-Campbell-Hausdorff formula, there exists a skew-symmetric bilinear map $\omega$ such that $(x, z)\cdot (x', z') = (x+x', z + z'+\omega(x,x'))$. Hence,
\begin{align} 
d(w_{1}^{-1},w_{2}^{-1}) &= \|w_{2} \cdot w_{1}^{-1}\|  = \|(x_{2} - x_{1}, z_{2} - z_{1} + \omega(x_{1}, x_{2})\| \notag \\
& \lesssim \|(x_{2} - x_{1},z_{2} - z_{1} - \omega(x_{1}, x_{2})\|+ \sqrt{|\omega(x_1, x_2)|} \notag \\
& = d(w_{1},w_{2}) + \sqrt{|\omega(x_1, x_2-x_1)|} \notag\\
& \lesssim d(w_{1},w_{2}) + \sqrt{d(w_{1},w_{2})}\sqrt{\|w_{1}\|} \label{form44}, 
\end{align}
where the last inequality is due to Cauchy-Schwartz inequality. It follows that $d(w_{1}^{-1},w_{2}^{-1}) \leq \kappa \|w_{1}\|$, if the constant $C$ is chosen large enough. Hence, the estimate \eqref{form43} is legitimate, and we may further use \eqref{form44} to obtain
\begin{displaymath} |K(w_{1}^{-1}) - K(w_{2}^{-1})| \lesssim \frac{\|w_{2}^{-1} \cdot w_{1}\|^{\beta}}{\|w_{1}\|^{Q-1 + \beta}} + \frac{\|w_{2}^{-1} \cdot w_{1}\|^{\beta/2}\|w_{1}\|^{\beta/2}}{\|w_{1}\|^{Q-1 + \beta}} \lesssim \frac{\|w_{2}^{-1} \cdot w_{1}\|^{\beta/2}}{\|w_{1}\|^{Q-1 + \beta/2}}, \end{displaymath}
as claimed. \end{proof}

\subsection{Intrinsic graphs}
\label{sec:ilgs}
In this section, we will be concerned with step-2 Carnot groups equipped with a homogeneous norm $\|\cdot\|$.
Intrinsic Lipschitz graphs were introduced in \cite{FSS1} as a model for low codimensional rectifiability in the Carnot group setting. Let $v \in \G$ be a horizontal vector and let $\bV$ be the subgroup it spans: $\{\delta_t(v): t \in \R\}$. We then let $\bW$ be the orthogonal complement of $v$ in $\G$. One can see that this is a vertical subgroup. Every $p \in \G$ can be uniquely written as a product $p_\bW p_\bV$. Thus, we naturally obtain the projection maps $\pi_\bV(p) = p_\bV$ and $\pi_\bW(p)=p_{\bW}$.

Given a function $\phi : \bW \to \R$, we can define its {\em intrinsic graph} as
\begin{align*}
  \Sigma(\phi) := \{ g \delta_{\phi(g)}(v) : g \in \bW \}.
\end{align*}
By abuse of notation, we will also interpret $\Sigma(\phi)$ as a function on $\bW$
$$\Sigma(\phi)(g):= g \delta_{\phi(g)}(v)$$
which parameterizes the graph. A function $\phi : \bW \to \R$ is said to be {\em linear} if there is a linear functional $T \in (\R^{n_1-1})^*$ so that $\phi(x) = T(\pi(x))$ for all $x \in \bW$. As this is a 1-1 correspondence, for any linear function $\phi : \bW \to \R$, we write $L_\phi \in (\R^{n_1-1})^*$ for the corresponding functional. Similarly, $\phi$ is {\em affine} if $T$ is affine. Then vertical hyperplanes are precisely the intrinsic graphs of affine functions.

Note that left translations of intrinsic graphs are still intrinsic graphs. Thus, given a $\phi : \bW \to \R$ and a $p \in \G$, we let $\phi^p : \bW \to \R$ be the function so that
\begin{align*}
  \Sigma(\phi^p) = p\Sigma(\phi).
\end{align*}
Given a $0 < L < 1$ and a homogeneous norm $\|\cdot\|$ on $\G$, we can define a double cone by
\begin{align*}
  C_L = \{ p \in \G : \|p\| < L^{-1} \|\pi_\bV(p)\|\}.
\end{align*}
We have that $C_L$ is decreasing as $L \to 1$. We now define an {\em $L$-intrinsic Lipschitz graph} to be an intrinsic graph $\Sigma = \Sigma_\phi$ for which
\begin{align*}
  p C_L \cap \Sigma = \{0\}, \qquad \forall p \in \Sigma.
\end{align*}
Any function $\phi : \bW \to \R$ that has an intrinsic Lipschitz graph is said to be intrinsic Lipschitz.

\begin{definition}
  A function $\phi : \bW \to \R$ with $\phi(\0) = 0$ is {\em intrinsically differentiable at \0} if there is some linear function $T : \bW \to \R$ so that
  \begin{align*}
    \lim_{g \to \0} \frac{|\phi(g) - T(g)|}{\|g\|} = 0.
  \end{align*}
  The {\em intrinsic differential} of $\phi$ is $d\phi_\0 = L_T$.

  A function $\phi : \bW \to \R$ is {\em intrinsically differentiable at $p \in \bW$} if, setting $g = p \delta_{\phi(p)}(v)$ and $\psi = \phi^{(g^{-1})}$, we have that $\psi$ is differentiable at \0. We then define the {\em intrinsic differential} of $\phi$ at $p$ as $d\phi_p = d\psi_\0$.
\end{definition}

It was proven in\cite{FranchiMarchiSerapioni14} that in Carnot groups of type $\star$, which include Carnot groups of step $2$, intrinsic Lipschitz functions are intrinsically differentiable at almost every $p \in \bW$.

We now metrize $(\R^{n_1-1})^*$ with the usual norm $\|\cdot\|_{\text{op}}$.

\begin{definition}
  A function $\phi : \bW \to \R$ is an {\em intrinsic $C^1(\bW)$ function} if $d\phi$ exists at every point $p \in \bW$ and is continuous. If $\phi$ is $C^1$, then so is $\phi^p$ for any $p \in \G$.

  For $\alpha \in [0,1]$, we now define the intrinsic $C^{1,\alpha}(\bW)$ functions as the functions of $C^1(\bW)$ for which there is a $H > 0$ so that
  \begin{align}
  \label{eq:c1alpha}
    \|d\phi^{(p^{-1})}_w - d\phi^{(p^{-1})}_\0\|_{\text{op}} \leq H \|w\|^\alpha, \qquad \forall p \in \Sigma(\phi), w \in \bW.
  \end{align}
\end{definition}

We can then define the {\em intrinsic gradient} as the $\nabla^\phi \phi(p) \in \R^{n_1-1}$ so that $d\phi^{(p^{-1})}_\0(v) = \nabla^\phi \phi(p) \cdot \pi(v)$.

\begin{remark}
\label{remk:c1alip}If $\phi \in C^1(\bW)$ and $\nabla^\phi \phi \in L^\infty$, then $\phi \in C^{1,0}(\bW)$. It is then known that $\phi$ is also intrinsic Lipschitz. This fact is essentially \cite[Lemma 2.22]{CFO19} whose proof for the Heisenberg group works more broadly for step-2 Carnot groups once we replace the use of their Proposition 2.23 with our Proposition \ref{holder}. In particular, if $\phi \in C^1(\bW)$ satisfies the conditions of \eqref{infinity}, then $\phi$ is intrinsic Lipschitz.
\end{remark}

We will also need the following area formula for functions $\phi\in C^1(\mathbb{W})$ which was recently obtained by Di Donato.
\begin{lemma}\cite[Proposition 5.10]{DiDonato}
Let $G=(\R^N, \cdot)$ be a step-2 Carnot group with homogeneous dimension $Q$. There exists a left-invariant homogeneous distance $d_{1}$ on $\mathbb{G}$ such that the associated spherical Hausdorff measure $\mathcal{S}^{Q-1} = \mathcal{S}^{Q-1}_{d_{1}}$ satisfies
\begin{equation}\label{eq:area_form}
\int_{\Sigma} h \;d\mathcal{S}^{Q-1}= \int_{\mathbb{W}} (h \circ \phi) \sqrt{1+|\nabla^{\phi}\phi|^2}\, d\mathcal{L}^{N-1},
\end{equation}
for every vertical hyperplane $\bW$, every $\phi \in C^{1}(\mathbb{W})$ with intrinsic graph $\Sigma$, and every $h\in L^1(\mathcal{S}^{Q-1}|_{\Sigma})$.
\end{lemma}

\section{Affine approximation of $C^{1,\alpha}$ intrinsic graphs}\label{curves}
The goal of this section is to prove a H\"older estimate for the affine approximation of intrinsic $C^{1,\alpha}$ graphs. We first lay down some necessary notation. Recall that $Q$ denotes the homogeneous dimension of the step--$2$ Carnot group $\mathbb{G}$, and the first layer of $\text{Lie}\,(\mathbb{G})$ is of dimension $m$. Fix an arbitrary nonzero vector $X_1$ in the first layer of $\text{Lie}\,(\mathbb{G})$, and fix arbitrary vectors $X_2, \ldots, X_m$ such that 
$$
X_1, \ldots, X_m
$$ 
is a basis of the first layer of $\text{Lie}\,(\mathbb{G})$. Let $T_1, \ldots, T_n$ be a basis of 
$$
\text{span}\,\{[X_1,X_2], [X_1, X_3], \ldots, [X_1, X_m]\}.
$$
Let $S_1, \ldots, S_k$ be vectors such that 
$$
\{T_1, \ldots, T_n\} \cup \{S_1, \ldots, S_k\}
$$
is a basis of the second layer of $\text{Lie}\,(\mathbb{G})$, and such that each of $S_1, \ldots, S_k$ is a linear combination of 
$$
\{[X_i,X_j]:2\leq i,j\leq m, i\neq j\}.
$$
Note the set $\{S_1, \ldots, S_k\}$ may be empty. We have $Q=m+2n+2k$. The the domain of our intrinsic Lipschitz graphs will be the subgroup of $\G$ generated by the vector fields $\{X_2, \ldots, X_m, T_1, \ldots, T_n\}$. 

Every point of $\mathbb{G}$ can be written as
$$
e^{\sum_{i=1}^m x_i X_i + \sum_{i=1}^n t_i T_i + \sum_{i=1}^k s_iS_i}.
$$
We use the coordinates $(x_1, \ldots, x_m, t_1, \ldots, t_n, s_1, \ldots, s_k)$ to denote this point. We have that 
$$
\{(0, x_2, \ldots, x_m, t_1, \ldots, t_n, s_1, \ldots, s_k)\}
$$ 
is a co-dimensional $1$ vertical subgroup, which we call vertical hyperplane $\mathbb{W}$. Denote each point on $\mathbb{W}$ by its coordinate $(x_2, \ldots, x_m, t_1, \ldots, t_n, s_1, \ldots, s_k) =(x,t,s)$, where $x:=(x_2, \ldots, x_m)$. In contrast to Section \ref{sec:prelim}, from now on we denote $(x,\zeta)$ by $(x,t,s)$ and denote $|t|:= \sqrt{\sum t_i^2}, |s|:= \sqrt{\sum s_i^2}$. We have
$$
\|z\| \sim  |x|+|t|^{1/2}+|s|^{1/2}.
$$

For convenience, we denote $T=(T_1, \ldots, T_n)$ and $S=(S_1, \ldots, S_k)$. We define constants:
\begin{align*}
&[X_j, X_1]= 2\sum_{i=1}^n C_{j1}^{t_i}T_i =: 2C_{j1}\cdot T, \quad \forall 2\leq j\leq m,\\
&[X_i, X_j] = 2\sum_{l=1}^n C_{ij}^{t_l}T_l + 2\sum_{l=1}^k C_{ij}^{s_l}S_l =: 2C_{ij}^t\cdot T + 2C_{ij}^s \cdot S, \quad \forall 2\leq i,j \leq m,
\end{align*}
where the $C_{j1}$ and the $C_{ij}^t$ are fixed vectors in $\mathbb{R}^n$, and the $C_{ij}^s$ are fixed vectors in $\mathbb{R}^k$. We have $C_{ij}^t = -C_{ji}^t, C_{ij}^s = -C_{ji}^s$ for $2\leq i,j \leq m$. Note the relation between the $C_{ij}$ and the bilinear map $w$ in Section \ref{sec:prelim} is 
$$
w(x, x') = \Big(-C_{j1}x_1x_j' + C_{j1} x_jx_1' + \sum_{i,j=2}^m C_{ij}^t x_i x_j', \sum_{i,j=2}^m C_{ij}^s x_i x_j'\Big) \in \mathbb{R}^{n+k}.
$$

We obtain the relations between the $\partial_{x_l}$ and the $X_l$ from the Campbell-Hausdorff formula. For $2\leq l \leq m$, consider a curve in $\mathbb{G}$
\begin{align*}
\gamma(\epsilon) &= e^{\epsilon X_l}(x_1, x, t,s) = e^{x_1X_1 + x\cdot X + t \cdot T + s\cdot S}e^{\epsilon X_l}\\
&= e^{x_1X_1 + x\cdot X + \epsilon X_l + (t- \sum_{i=2}^mC_{li}^t\epsilon x_i -C_{l1}\epsilon x_1)\cdot T + (s - \sum_{i=2}^m C_{li}^s \epsilon x_i)\cdot S}\\
&=\big(x_1, \cdots, x_l+\epsilon, \cdots, x_m, t- \sum_{i=2}^mC_{li}^t\epsilon x_i -C_{l1}\epsilon x_1, s - \sum_{i=2}^m C_{li}^s \epsilon x_i\big).
\end{align*}
We have
\begin{equation}\label{derivative}
\begin{aligned}
\gamma'(0)= X_l(x_1, x,t,s) = \partial_{x_l}+(\sum_{i=2}^m C_{il}^t x_i + C_{1l}x_1) \cdot \partial_t + \sum_{i=2}^m C_{il}^s x_i \cdot \partial_s.
\end{aligned}
\end{equation}
Similarly, we have $T_l = \partial_{t_l}, S_l = \partial_{s_l}$.

For the main result in this section we will need the following existence lemma. It was proved in \cite[Lemma 4.5]{CFO19} for the case of the Heisenberg group and it can be generalized with straightforward modifications in the case of step-$2$ Carnot groups. 

\begin{lemma}\label{l:integral_along_curves}
Assume that $\phi \in C^{1}(\mathbb{W})$, and fix any $x=(x_2, \ldots, x_m) \in \mathbb{R}^{m-1}$. Then, for all $(x_0,t_0,s_0)\in\mathbb{W}$ and every integral curve $\gamma(\cdot): \mathbb{R} \to \mathbb{W}$ satisfying
\begin{displaymath}
\left\{\begin{array}{ll}\frac{d}{du}\gamma(u)=\sum_{i=2}^m x_i Y_i(\gamma(u)),&u\in \mathbb{R},\\\gamma(u_0)=(x_0,t_0,s_0),\end{array} \right.
\end{displaymath}
one has $\frac{d}{du} \phi(\gamma(u)) = \nabla^\phi \phi(\gamma(u))\cdot x$ for every $u\in \mathbb{R}$, 
and in particular,
\begin{displaymath}
\phi(\gamma(u)) = \int_{u_0}^u \nabla^\phi \phi(\gamma(v))\cdot x\,dv + \phi(x_0,t_0,s_0),\quad u\in  \mathbb{R}.
\end{displaymath}
\end{lemma}

\begin{proof}  
Fix $(x_0,t_0,s_0)\in\mathbb{W}$. By our assumption, $\phi$ is intrinsic differentiable and $\nabla^{\phi}\phi$ is continuous. Therefore,  \cite[Theorem 4.95]{MR3587666} implies that the function $u \mapsto \phi(\gamma(u))$ is in $C^{1}$, and 
\begin{displaymath} 
\phi(\gamma(u)) = \int_{u_0}^u \nabla^\phi \phi(\gamma(v))\cdot x\,dv + \phi(x_0,t_0,s_0),\quad u\in  \mathbb{R}. \end{displaymath}
Although \cite[Theorem 4.95]{MR3587666} is stated for the Heisenberg group, the previous argument only utilizes the implication (iii) $\implies$ (ii) (from \cite[Theorem 4.95]{MR3587666}) whose proof remains valid in our context with no significant changes.
\end{proof}

We are now ready to state and prove the following H\"older estimate  which is the main result in this section.
\begin{proposition}\label{holder}
Let $\phi\in C^{1,\alpha}(\mathbb{W})$ with $L:=\|\nabla^\phi \phi\|_\infty <\infty$. Then for any $p_0=(0,x^0, t^0,s^0)\cdot (\phi(x^0, t^0,s^0),0,0,0)\in \Sigma(\phi)$
$$
|\phi^{(p_0^{-1})}(x,t,s)-\nabla^\phi \phi(x^0, t^0,s^0)\cdot x|\lesssim \max\{\|(x,t,s)\|^{1+\alpha}, \|(x,t,s)\|^{1+\frac{\alpha}{2}}\}, \quad \forall (x,t,s)\in \mathbb{W},
$$
where the implicit constant depends on $L$, and the H\"older continuity constant in the definition of $C^{1,\alpha}(\mathbb{W})$.
\end{proposition}

\begin{proof}
By the definition of $\nabla^\phi \phi$, we have
$$
|\phi^{(p_0^{-1})}(x,t,s)-\nabla^\phi \phi(x^0, t^0,s^0)\cdot x| = |\phi^{(p_0^{-1})}(x,t,s)-\nabla^{\phi^{(p_0^{-1})}}\phi^{(p_0^{-1})}(0, 0,0)\cdot x|.
$$
And since the constant $L$ and the H\"older continuity constant are invariant under left translations, it suffices to show for $\phi$ with $\phi(0,0,0)=0$,
$$
|\phi(x,t,s)-\nabla^\phi \phi(0,0,0)\cdot x|\lesssim \|(x,t,s)\|^{1+\frac{\alpha}{2}},
$$
for $(x,t,s)$ with $\|(x,t,s)\|$ sufficiently small.

To estimate the difference between $\phi(x,t,s)$ and its linear approximation $\nabla^\phi \phi(0,0,0)\cdot x$, we will introduce finitely many intermediate points for $(x,t,s)$ and $(0,0,0)$, which allow us to estimate the differences in $\phi$ values at successive points. For that, we now construct vector fields whose integral curves will guide the selection of these intermediate points.

In order to define appropriate vector fields $Y_2, \ldots, Y_m$ on which the integral curves travel along, we need to compute $\phi^{(p^{-1})}$. We have
\begin{align*}
\Sigma(\phi)(x,t,s)&= (0, x,t,s) \cdot (\phi(x,t,s),0,0,0).
\end{align*}
Let $p =(0,\bar x, \bar t,\bar s)\cdot \big(\phi(\bar x, \bar t,\bar s), 0, 0, 0\big)$. By the Campbell-Hausdorff formula,
\begin{align*}
&\quad p^{-1}\cdot \Sigma(\phi)(x',t',s')\\
& = \big(-\phi(\bar x, \bar t, \bar s), 0,0,0
\big) \cdot (0, -\bar x, -\bar t, -\bar s)\cdot (0,x',t',s') \cdot \big(\phi(x',t',s'),0,0,0\big)\\
&=\big(-\phi(\bar x, \bar t, \bar s), 0,0,0
\big) \cdot \big(0, x'-\bar x, t'-\bar t - \sum_{i,j=2}^m C_{ij}^t\bar x_ix_j', s'-\bar s - \sum_{i,j=2}^m C_{ij}^s\bar x_ix_j'\big)\\
&\qquad \qquad \qquad \qquad \qquad \qquad \cdot \big(\phi(x',t',s'),0,0,0\big)\\
& =\Big(\phi(x',t',s')-\phi(\bar x, \bar t, \bar s), x'-\bar x,\\
&\qquad t'-\bar t - \sum_{i,j=2}^m C_{ij}^t\bar x_ix_j' + \sum_{i=2}^m C_{i1} \big(\phi(\bar x, \bar t,\bar s)+\phi(x', t',s')\big)(x_i'-\bar x_i), s'-\bar s - \sum_{i,j=2}^m C_{ij}^s\bar x_ix_j'\Big)\\
& = \Big(\phi^{(p^{-1})}(x, t, s), x, t + \sum_{i=2}^m C_{i1} x_i \phi^{(p^{-1})}(x, t, s), s\Big)\\
&=(0,x, t, s) \cdot \big(\phi^{(p^{-1})}(x, t, s), 0,0,0\big),
\end{align*}
where
$$
\left\{
\begin{aligned}
&\phi^{(p^{-1})}(x, t, s) = \phi(x',t',s')-\phi(\bar x, \bar t, \bar s),\\
& x'= \bar x + x,\\
& t' = \bar t + t + \sum_{i,j=2}^m C_{ij}^t \bar x_ix_j - 2\sum_{i=2}^m C_{i1} x_i \phi(\bar x, \bar t, \bar s),\\
& s' = \bar s + s + \sum_{i,j=2}^m C_{ij}^s \bar x_i x_j.
\end{aligned}
\right.
$$
  For the formulas for $t'$ and $s'$, we used the fact that $x_j' = \bar x_j + x_j$ and the fact that the $C_{ij}$ are the coefficients of an antisymmetric form to get that
  \begin{align*}
    \sum_{i,j=2}^m C_{ij}^t \bar x_i x_j' = \sum_{i,j=2}^m C_{ij}^t \bar x_i x_j.
  \end{align*}
Therefore
\begin{align*}
\phi^{(p^{-1})}(x,t,s) = \phi \Big(x+\bar x, t+\bar t + \sum_{i,j=2}^m C_{ij}^t \bar x_ix_j - 2\sum_{i=2}^m C_{i1} x_i \phi(\bar x,\bar t, \bar s), s+\bar s + \sum_{i,j=2}^m C_{ij}^s \bar x_ix_j\Big) - \phi(\bar x,\bar t,\bar s).
\end{align*}

We define vector fields on the $(x,t,s)$-hyperplane as follows:
$$
Y_l = \partial_{x_l} + \big(\sum_{i=2}^m C_{il}^tx_i - 2C_{l1}\phi(x,t,s)\big)\cdot \partial_t + \big(\sum_{i=2}^m C_{il}^s x_i\big)\cdot \partial_s, \quad \forall 2\leq l \leq m.
$$
This comes from the group multiplication law. For example, in the case of the $n$-dimensional Heisenberg group $\mathbb{H}^n$,
\begin{align*}
&\quad \left(x_1, \ldots, x_n, y_1, \ldots, y_n, t\right) \cdot \left(x_1', \ldots, x_n', y_1', \ldots, y_n',t'\right) \\
&= \left(x_1 + x_1', \ldots, x_n+ x_n', y_1 + y_1', \ldots, y_n + y_n', t+t' + \frac{1}{2}\big(x_1 y_1' - x_1'y_1 + \ldots, x_n y_n' - x_n'y_n\big)\right),
\end{align*}
so the vector fields take the form:
\begin{align*}
Y_i &= \left(0,\ldots, 0,\underbrace{1}_{\text{$i$-th coordinate}}, 0\ldots, 0,\frac{1}{2}y_i\right), \quad \forall 2\leq i\leq n,\\
Y_{n+1} &= \left(0, \ldots, 0, \underbrace{1}_{\text{$(n+1)$-th coordinate}}, 0, \ldots, 0, -\phi(x,y,t)\right),\\
Y_i &= \left(0, \ldots, 0, \underbrace{1}_{\text{$(n+i)$-th coordinate}}, 0, \ldots, 0, \frac{1}{2} x_i \right), \quad \forall n+2 \leq i \leq 2n.
\end{align*}

By the definition of $\nabla^\phi \phi$, we have $Y_l(\phi) = (\nabla^\phi \phi)_l$, the $l$-th component of the $\nabla^\phi \phi$. Due to (\ref{derivative}), by viewing the $Y_l$ as vector fields on $\mathbb{G}$,
\begin{align*}
Y_l &= X_l + \big(C_{l1}x_1 - 2C_{l1}\phi(x,t,s)\big)\cdot T.   
\end{align*}
Thus for $2\leq p, q\leq m$,
\begin{align*}
[Y_p, Y_q] &=[X_p, X_q] + Y_p\big(C_{q1}x_1-2C_{q1}\phi\big)\cdot T - Y_q\big(C_{p1}x_1-2C_{p1}\phi\big)\cdot T\\
&= [X_p, X_q]-(\nabla^\phi \phi)_p\cdot  (2C_{q1}\cdot T) + (\nabla^\phi \phi)_q\cdot  (2C_{p1}\cdot T).
\end{align*}
Therefore for each $1\leq i \leq k$, $S_i$ is a linear combination of $[Y_p, Y_q] \,(2\leq p, q\leq m)$, modulo a linear combination of $T_1, \ldots, T_n$.

Let $\gamma_1:[0,1] \to \mathbb{R}^{m-1}\times \mathbb{R}^n \times \mathbb{R}^k$ be an absolutely continuous curve such that $\gamma_1(0)=(0,0,0)$ and
$$
\gamma_1'(u) = \sum_{i=2}^m x_i Y_i(\gamma_1(u)), \quad  \forall u\in [0,1].
$$
Since $\phi$ is intrinsically differentiable, $\phi$ is continuous. So, by Peano's theorem, this integral curve exists (not necessarily unique) on some maximal interval $J$ containing $0$. Moreover, 
$$
\gamma_1(u)= \int_0^u \sum_{i=2}^m x_iY_i(\gamma_1(v))\,dv=\Big(x_2u,\ldots, x_mu, \underbrace{-2\sum_{i=2}^m C_{i1}^t x_i \int_0^u \phi(\gamma_1(v))\,dv}_{t \text{ components}},\underbrace{\vphantom{\left(\frac{a^{8}}{b}\right)} 0}_{s\text{ components}}\Big)
$$ 
leaves any compact set of the plane $\mathbb{W}$, as $u$ tends to either endpoint of $J$; see for instance
\cite[Corollary 2.16]{MR2961944} or \cite[Theorem 2.1]{MR587488}.
We presently want to argue that $J = \R$. 

By Lemma \ref{l:integral_along_curves}, we have an integral representation of $\phi$ along $\gamma_1$:
\begin{equation}\label{form26} \phi(\gamma_1(u)) = \int_{0}^u \nabla^{\phi}\phi(\gamma_1(v))\cdot x\;dv, \qquad u \in J. \end{equation}
Hence, for $y \in J$, each $t_k$-component of $\gamma_1(y)$ equals to
\begin{equation}\label{form27} 
-2\sum_{i=2}^m C_{i1}^{t_k} x_i \int_0^y \phi(\gamma_1(u))\,du 
\stackrel{\eqref{form26}}{=} -2\sum_{i=2}^m C_{i1}^{t_k} x_i\int_{0}^{y} \int_{0}^{u} \nabla^{\phi}\phi(\gamma_1(v))\cdot x \, dv \, du.   
\end{equation}
Since $|\nabla^{\phi}\phi| \leq L$ by hypothesis, we see that the $t$-components of $\gamma_1(y)$ cannot blow up in finite time; hence $J = \R$, and \eqref{form27} even provides a quantitative bound:  each $t_k$-component of $\gamma_1(y)$ satisfies $\lesssim_L y^2$, for all $y\in J=\mathbb{R}$. A similar argument, combined with Lemma \ref{l:integral_along_curves}, also supports our assertions for the curves $\gamma_2, \alpha$, and $\beta$ constructed in the remainder of this section.

Note that the first component of $\gamma_1(1)$ is $x$. By Lemma \ref{l:integral_along_curves}, we have
\begin{align*}
&\quad \big|\phi(\gamma_1(1)) - \nabla^\phi \phi(0,0,0)\cdot x\big|\\
&= \Big| \int_0^1 \sum_{i=2}^m x_i Y_i(\phi)(\gamma_1(u))\,du -  \nabla^\phi \phi(0,0,0)\cdot x\Big|\\
& = \Big| \int_0^1 \big(\nabla^\phi \phi(\gamma_1(u)) - \nabla^\phi \phi(0,0,0)\big)\cdot x\,du\Big|\\
&\lesssim |x| \int_0^1 \|\gamma_1(u)\|^\alpha\,du,
\end{align*}
where the implicit constant depends on the H\"older continuity constant. By Gronwall's inequality, $\gamma_1(u)$ is bounded for $u\in [0,1]$. Hence
\begin{align*}
\gamma_1(u)&= \int_0^u \gamma_1'(v)\,dv = \int_0^u \sum_{i=2}^m x_i Y_i\Big|_{\gamma_1(v)}\,dv\\
&= \int_0^u \sum_{i=2}^m x_i \partial_{x_i}\,dv + \int_0^u \int_0^v\sum_{j=2}^mx_jY_j\big(\sum_{i=2}^m x_iY_i\big)\Big|_{\gamma_1(w)}\,dw\,dv\\
&= \int_0^u \sum_{i=2}^m x_i \partial_{x_i}\,dv + \sum_{i,j=2}^mx_j x_i\int_0^u \int_0^vY_j\big(Y_i\big)\Big|_{\gamma_1(w)}\,dw\,dv,
\end{align*}
where the first term is $O(|x|)$ and is along the first layer, and the second term, which reflects the size on the second layer, is $O(|x|^2)$. Therefore $\|\gamma_1(u)\|=O(|x|)$, and thus $|\phi(\gamma_1(1)) - \nabla^\phi \phi(0,0,0)\cdot x| \lesssim |x|^{1+\alpha} \lesssim \|(x,t,s)\|^{1+\frac{\alpha}{2}}$ for $\|(x,t,s)\|$ sufficiently small.

Denote $\gamma_1(1)=(x,t^1, s^1)$. We have $|\phi(\gamma_1(1))|\lesssim \|(x,t,s)\|$ for $\|(x,t,s)\|$ sufficiently small. We have $|s-s^1| \leq |s|+|s^1|=O(\|(x,t,s)\|^2)$. Thus we can travel within $O(\|(x,t,s)\|^2)$ time from $(x,t^1, s^1)$ to $(x,t^1,s)$ along a fixed linear combination of vector fields $S_1, \ldots, S_k$. Hence we can travel within $O(\|(x,t,s)\|^2)$ time from $(x,t^1, s^1)$ to $(x,t',s)$ for some $t'\in \mathbb{R}^n$, along a fixed linear combination of vector fields $[Y_p, Y_q] \,(2\leq p, q\leq m)$. By Chow's theorem, if we travel along $Y_2, \ldots, Y_m$, we can travel from $(x,t^1, s^1)$ to $(x,t',s)$ within $O(\|(x,t,s)\|)$ time. Let  $\gamma_2: [0,1] \to \mathbb{R}^{m-1} \times \mathbb{R}^n \times \mathbb{R}^k$ be an absolutely continuous curve such that $\gamma_2(0)=(x,t^1, s^1), \gamma_2(1)=(x,t',s)$, and such that 
$$
\gamma_2'(u) = \sum_{i=2}^m b_i(u)Y_i(\gamma_2(u)), \quad a.e.\,u\in [0,1],
$$
where the functions $b_i(u)=O(\|(x,t,s)\|)$ are piecewise constant. Thus, by viewing $\gamma_2$ as a concatenation of several connecting curves, we can still apply Lemma \ref{l:integral_along_curves}.

By Gronwall's inequality, $\gamma_2(u)$ is bounded for $u\in [0,1]$. Hence
\begin{align*}
\gamma_2(u)-\gamma_2(0) &= \int_0^u \gamma_2'(v)\,dv = \int_0^u \sum_{i=2}^m b_i(v)Y_i\Big|_{\gamma_2(v)}\,dv\\
&=\int_0^u \sum_{i=2}^m b_i(0)Y_i(\gamma_1(1))\,dv + \int_0^u\int_0^v \sum_{j=2}^mb_j(w)Y_j\big(\sum_{i=2}^mb_i(w)Y_i\big)\Big|_{\gamma_2(w)}\,dw\,dv 
\end{align*}
where the first term is $O(\|(x,t,s)\|)$, the second term is $O(\|(x,t,s)\|^2)$. By the definition of $Y_i(\gamma_1(1))$, the $t, s$ components of the above first term are $O(\|(x,t,s)\|^2)$. Therefore $\|\gamma_2(u)\| = O(\|(x,t,s)\|)$. Hence
\begin{align*}
&\quad \big|\phi(\gamma_2(1)) - \phi(\gamma_2(0))- \int_0^1\sum_{i=2}^m b_i(u) (\nabla^\phi \phi)_i(0,0,0)\,du\big|\\
& =\big| \int_0^1 \sum_{i=2}^m b_i(u) \big((\nabla^\phi \phi)_i(\gamma_2(u))-(\nabla^\phi \phi)_i(0,0,0)\big)\,du\big|\\
&\lesssim \|(x,t,s)\|^{1+\alpha} \lesssim \|(x,t,s)\|^{1+\frac{\alpha}{2}},
\end{align*}
for sufficiently small $\|(x,t,s)\|$. And since the vector $\big(\int_0^1 b_i(u)\,du\big)_{2\leq i\leq m}$ equals the $x$ component of 
$$
\int_0^1 \sum_{i=2}^m b_i(u)Y_i(\gamma_2(u))\,du = \gamma_2(1)-\gamma_2(0) =(0,t'-t^1,s-s^1),
$$
we have $\int_0^1 b_i(u)\,du=0$ for $2\leq i \leq m$. Therefore
$$
\big|\phi(\gamma_2(1)) - \phi(\gamma_2(0))\big| \lesssim \|(x,t,s)\|^{1+\frac{\alpha}{2}},
$$
for sufficiently small $\|(x,t,s)\|$.

Recall $\|(x,t', s)\|=\|\gamma_2(1)\| =O(\|(x,t,s)\|)$. It remains to show 
$$
|\phi(x,t,s)-\phi(x,t',s)|\lesssim \|(x,t,s)\|^{1+\frac{\alpha}{2}},
$$
for sufficiently small $\|(x,t,s)\|$. Without loss of generality assume $\phi(x,t,s)-\phi(x,t',s)\neq 0$. We have
$$
Y_l(x,t,s)-Y_l(x,t', s)= -2\big(\phi(x,t,s)-\phi(x,t',s)\big)C_{l1}\cdot T.
$$
There exist constants $\theta_2, \ldots, \theta_m$ such that
$$
2\sum_{l=2}^m \theta_l C_{l1}\cdot T = \text{sgn}\,(\phi(x,t,s)-\phi(x,t',s))(t-t')\cdot T.
$$
We have $|\theta_2|+ \ldots + |\theta_m|\lesssim |t-t'|$.

Let $\alpha, \beta: [0,\infty)\to \mathbb{R}^{m-1} \times \mathbb{R}^n \times \mathbb{R}^k$ be absolutely continuous curves such that for a.e. $u$,
$$
\alpha(0)= (x,t,s),\quad  \beta(0) = (x,t',s), \quad \alpha'(u) =\sum_{l=2}^m \theta_l Y_l(\alpha(u)), \quad \beta'(u) = \sum_{l=2}^m \theta_l Y_l(\beta(u)).
$$
Assume for contradiction that $|\phi(\beta(u))-\phi(\alpha(u))|>|t'-t|^{\frac{1}{2}+\frac{\alpha}{4}}$ for all $u\in [0,\infty)$. Then the sign of $\phi(\beta(u))-\phi(\alpha(u))$ does not change; otherwise $\phi(\beta(u))-\phi(\alpha(u))$ achieves $0$ somewhere in $[0, \infty)$. We have
\begin{align*}
\beta'(u)-\alpha'(u) =-2(\phi(\beta(u))-\phi(\alpha(u))) \sum_{l=2}^m \theta_l C_{l1}\cdot T = -|\phi(\beta(u))-\phi(\alpha(u))|(t'-t)\cdot T,
\end{align*}
and thus
\begin{align*}
\beta(u)-\alpha(u)&= (0,t'-t,0) +\int_0^u (\beta'(v)-\alpha'(v))\,dv\\
&=(t'-t)\cdot T \Big(1-\int_0^u|\phi(\beta(v))-\phi(\alpha(v))|\,dv\Big)
\end{align*}
can reach $0$ for some $u>0$; contradiction.

Take the smallest $u\geq 0$ such that $|\phi(\beta(u))-\phi(\alpha(u))|\leq |t'-t|^{\frac{1}{2}+\frac{\alpha}{4}}$. Then on $[0,u]$, $\phi(\beta(v))-\phi(\alpha(v))$ does not change sign, and $\sum_{l=2}^m \theta_l\big(Y_l(\beta(v))-Y_l(\alpha(v))\big)$ does not change direction, and hence
\begin{align*}
&\quad |\phi(x,t',s)-\phi(x,t,s)|^2 |(t'-t)\cdot T|^2 = |\phi(\beta(u))-\phi(\alpha(u))|^2 |(t'-t)\cdot T|^2 - \int_0^u\frac{d}{dv} \big|\beta'(v)-\alpha'(v)\big|^2\,dv,
\end{align*}
and
\begin{align*}
&\quad \Big|\int_0^u\frac{d}{dv} \big|\beta'(v)-\alpha'(v)\big|^2\,dv\Big| = 2\Big| \int_0^u (\beta'(v)-\alpha'(v)) \cdot (\phi(\beta(v))-\phi(\alpha(v)))'((t'-t)\cdot T)\,dv\Big|\\
&= 2\Big|\int_0^u \sum_{l=2}^m \theta_l\big(Y_l(\beta(v))-Y_l(\alpha(v))\big) ((t'-t)\cdot T) \theta \cdot \big(\nabla^\phi \phi(\beta(v))-\nabla^\phi \phi(\alpha(v))\big)\,dv\Big|\\
&\lesssim |\theta| \sup_{v\in [0,u]} \|\beta(v)-\alpha(v)\|^\alpha \Big|\int_0^u \sum_{l=2}^m \theta_l \big(Y_l(\beta(v))-Y_l(\alpha(v))\big)((t'-t)\cdot T)\,dv\Big|\\
&\lesssim |t'-t| \cdot |t'-t|^{\frac{\alpha}{2}} \Big|\big((\beta(u)-\alpha(u))-(\beta(0)-\alpha(0))\big)((t'-t)\cdot T)\Big|\\
&\lesssim |t'-t| \cdot |t'-t|^{\frac{\alpha}{2}} \cdot |t'-t|^2.
\end{align*}
Therefore
\begin{align*}
|\phi(x,t',s)-\phi(x,t,s)| \lesssim |t'-t|^{\frac{1}{2}+\frac{\alpha}{4}} \lesssim \|(x,t,s)\|^{1+\frac{\alpha}{2}}.
\end{align*}
\end{proof}

\begin{remark}
If the collection $\{S_1, \ldots, S_k\}$ is empty, the construction of the curve $\gamma_2$ becomes unnecessary, and the proof proceeds without this step.
\end{remark}

\begin{remark}
In the above proof, we do not fully use the differentiability assumption that
\begin{equation}\label{strong}
\phi^{(p_0^{-1})}(x,t,s) =\nabla^\phi \phi(x^0,t^0,s^0)\cdot x + o(\|(x,t,s\|).
\end{equation}
Instead, of the differentiability assumption we use only 
\begin{equation}\label{weak}
\phi^{(p_0^{-1})}(x,0,0) =\nabla^\phi \phi(x^0,t^0,s^0)\cdot x + o(|x|).
\end{equation}
Thus (\ref{weak}) implies, by the above proposition, the seemingly stronger notion of differentiability (\ref{strong}), under the H\"older continuity condition $\nabla^\phi \phi \in C^\alpha$.
\end{remark}

\section{Proof of Theorem \ref{theoreminfinity}}\label{4}

In this section we prove Theorem \ref{theoreminfinity}. The proof follows along the lines of \cite[Section 3]{CFO19}. We provide an outline of the steps which are essentially the same as the ones from \cite{CFO19} and we specify what needs to be modified. Note, that in our theorem, unlike  the result from \cite{CFO19}, we do not require that the intrinsic graphs are compactly supported. We provide full details for the proof in the non-compact case.  


Similarly to \cite[Lemma 3.1]{CFO19}, proving Theorem \ref{theoreminfinity} for any $(Q-1)$-ADR measure $\mu$ on $\Sigma$ is equivalent to proving the theorem merely for $\mu=\mathcal{S}^{Q-1}$. Thus, we need to prove
\begin{equation}\label{3.1}
\|T_{\mu, \epsilon}f\|_{L^2(\mu)} \lesssim \|f\|_{L^2(\mu)}, \quad \forall f\in L^2(\mu), \epsilon>0
\end{equation}
for $\mu = \mathcal{S}^{Q-1}$, where the implicit constant is independent of $f$ and $\epsilon$. We now fix $\mu = \mathcal{S}^{Q-1}$.

A version of the $T1$ theorem needs to be employed, and for that, we need to introduce the Christ cubes as below (\cite{christ}). For $j\in \mathbb{Z}$, there exists a family $\Delta_j$ of disjoint subsets of $\Sigma$ with following properties:
\begin{itemize}
    \item 
    $\Sigma \subset \bigcup_{Q\in \Delta_j} \overline{Q}$.
    \item 
    If $j\leq k, Q\in \Delta_j$ and $Q'\in \Delta_k$, then either $Q\cap Q'=\emptyset$ or $Q\subset Q'$.
    \item 
    If $Q\in \Delta_j$, then $\text{diam}\,Q\leq 2^j =: \ell (Q)$.
    \item 
    Every cube $Q\in \Delta_j$ contains a ball $B(z_Q, c2^j) \cap \Sigma$ for some $z_Q\in Q$, and some constant $c>0$.
    \item 
    Every cube $Q\in \Delta_j$ has thin boundary: there is a constant $D\geq 1$ such that $\mu(\partial_\rho Q) \leq D \rho^{\frac{1}{D}} \mu(Q)$, where
    $$
\partial_\rho Q:= \{q\in Q: \text{dist}\,(q, \Sigma\backslash Q)\leq \rho \cdot \ell (Q)\}, \quad \rho>0.
    $$
\end{itemize}
The sets in $\Delta:= \cup \Delta_j$ are called Christ cubes. It is immediate from the construction that $\mu(Q) \sim \ell (Q)^{Q-1}$ for $Q\in \Delta_j$.

By the $T1$ theorem of David and Journé, applied to the homogeneous metric measure space $(\Sigma, d, \mu|_{\Sigma})$ (see Fernando's honors thesis \cite{surath} for the details on how the $T1$ theorem extends to the case of homogeneous metric measure space), in order to prove \eqref{3.1} it suffices to verify:
\begin{equation}\label{restricted}
\|T_{\mu, \epsilon} \chi_R\|_{L^2(\mu|R)}^2 \lesssim \mu(R), \quad \|T_{\mu, \epsilon}^* \chi_R\|_{L^2(\mu|R)}^2 \lesssim \mu(R),
\end{equation}
for all $R\in \Delta$, where $T_{\mu, \epsilon}^*$ is the formal adjoint of $T_{\mu, \epsilon}$, and the implicit constants are independent of $\epsilon$ and $R$. Since the inequality in (\ref{restricted}) for $T_{\mu, \epsilon}\chi_R$ and that for $T_{\mu, \epsilon}^*\chi_R$ are proved similarly, we will only focus on $T_{\mu, \epsilon}\chi_R$.

We now decompose the operator $T_{\mu, \epsilon}$. As in \cite{CL}, we fix a smooth even function $\psi:\R \to \R$ such that $\psi|_{[-1/2,1/2]}=1$, and $\psi|_{\R \setminus [-2,2]}=0$. We then consider the $\|\cdot \|$-radial functions $\psi_{j} : \G \to \R, j\in \mathbb{Z},$ defined by 
$$\psi_{j}(p)= \psi(\|\delta_{2^j}(p)\|)= \psi(2^{j}\|p\|),$$
and we let 
$$\eta_{j} := \psi_{j} - \psi_{j + 1}, j \in \mathbb{Z}.$$ 
Note that the functions $\{\eta_j\}_{j\in \mathbb{Z}}$ satisfy $\sum \eta_j =1$ and
$$
\text{supp}\,\eta_j \subseteq B(0, 2^{1-j})\backslash B(0, 2^{-2-j}).
$$
Denote
$$
T_{(j)}f(p) = \int K(q^{-1}\cdot p) \eta_j(q^{-1}\cdot p) f(q)\,d\mu(q), \quad S_N:= \sum_{j\leq N} T_{(j)}.
$$
Lemma 3.3 of \cite{CFO19} proves that $T_{\mu, \epsilon}f$ is close to $S_Nf$ in $L^\infty$ if $\epsilon\in [2^{-N}, 2^{-N+1})$ and that it suffices to show 
$$
\|S_N \chi_R\|_{L^2(\mu|R)}^2 \lesssim \mu(R), \quad \forall R\in \Delta, N\in \mathbb{Z},
$$
with the implicit constant independent of $R, N$. 

Denote 
$$
\rho(k) = 2^{1-k}/\ell(R), \quad \partial_{\rho(k)}R = \{q\in R: (\rho(k)/2)\cdot \ell (R)<\text{dist}\,(q,\Sigma\backslash R) \leq \rho(k) \cdot \ell (R)\}.
$$
Note that $\partial_{\rho(k)}R=\emptyset$ for $\rho(k)>2$. Same as in Section 3 of \cite{CFO19}, we let
\begin{align*}
&\quad S_N\chi_R \\
&= \sum_{2^{-N} \leq 2^{-j} < \min\{1, 2^{-k}\}} T_{(j)}\chi_R + \sum_{\min\{1, 2^{-k}\} \leq 2^{-j} <2^{-k}} T_{(j)}\chi_R + \sum_{2^{-k} \leq 2^{-j} < 4\ell (R)} T_{(j)}\chi_R + \sum_{2^{-j}\geq 4\ell (R)} T_{(j)}\chi_R\\
&=: S_{I}\chi_R + S_{II}\chi_R + S_{III}\chi_R + \sum_{2^{-j}\geq 4\ell (R)} T_{(j)}\chi_R.
\end{align*}
The estimates of $S_{I}\chi_R+ S_{III}\chi_R + \sum_{2^{-j}\geq 4\ell (R)} T_{(j)}\chi_R$ are obtained similarly as what was done in \cite[Section 3]{CFO19} and are not affected by the fact that we now consider $\phi$ without necessarily having compact support. In particular, the term $\sum_{2^{-j} \geq 4\ell(R)} T_{(j)} \chi_R$ is zero, because the kernel $q\mapsto K(q^{-1}\cdot p)\eta_j(q^{-1}\cdot p)$ of $T_{(j)}$ is supported in the complement of $B(p,2^{-2-j})$, where $2^{-j}\geq 4\ell(R)$.
The fact that $\phi$ might not be compactly supported  only affects the estimate of $S_{I\!I} \chi_R$. Note also that if $2^{-k} \leq 1$ then $S_{II}\chi_R=0$. Thus, we can assume that $2^{-k} > 1$.

It remains to estimate 
$$
S_{I\!I} \chi_R(p)= \sum_{j: 1\leq 2^{-j} < 2^{-k}} \int_\Sigma K (q^{-1}\cdot p) \eta_j(q^{-1}\cdot p)\,d\mu(q).
$$ 
We define $K_{I\!I} = \sum_{j : 1 \leq 2^{-j} < 2^{-k}} K \eta_j$.
Recall we assume that $\phi$ does not have rapid growth at $\infty$, i.e., there exist $0<\gamma, \theta<1$ such that for every $w=(x,t,s)\in \mathbb{W}, p\in \Sigma$,
$$
|\phi^{(p^{-1})}(x,t,s)| \lesssim \|(x,t,s)\|^{1-\theta}, \qquad |\nabla^\phi \phi (x,t,s)|\lesssim \|(x,t,s)\|^{-\gamma}.
$$
We need to show 
$$
\sum_{1< 2^{-k} \leq \ell (R)} \int_{\partial_{\rho(k)}R} |S_{I\!I} \chi_R(p)|^2\,d\mu(p) \lesssim \mu(R).
$$
It suffices to show $|S_{I\!I} \chi_R(p)|\lesssim 1$, for every $p\in \partial_{\rho(k)}R$ and every $1<2^{-k}\leq \ell(R)$.

Assume $\partial_{\rho(k)}R\neq \emptyset$ and fix $p=(0,x^0, t^0, s^0)\cdot (\phi(x^0, t^0, s^0),0,0,0) \in \partial_{\rho(k)}R$. Let 
\begin{align*}
\Phi_{p\mathbb{W}}(w)&=\Phi_{p\mathbb{W}}(x,t,s):=p\cdot \pi_\mathbb{W}(p^{-1} \cdot \Phi_\Sigma(w))\\
&=p\cdot \big(x-x^0, t-t^0- \sum_{i,j=2}^m C_{ij}^t x_i^0x_j + 2\sum_{i=2}^m C_{i1} (x_i-x_i^0)\phi(x^0, t^0, s^0), s-s^0 - \sum_{i,j=2}^m C_{ij}^s x_i^0x_j\big).
\end{align*}
Denote by $K_{I\!I}^*$ the kernel for the adjoint operator. By the annular boundedness
$$
\Big|\int_\mathbb{W} K_{I\!I} \big(\Phi_{p\mathbb{W}}(x,t,s)^{-1} \cdot p\big)\,dx\,dt\,ds\Big|\lesssim 1
.$$
Therefore,
\begin{align*}
|S_{I\!I}\chi_R(p)| &= \Big|\int_\Sigma K_{I\!I}(q^{-1}\cdot p)\,d\mu(q)\Big| \lesssim \Big|\int_{\mathbb{W}} K_{I\!I}\big(\Phi_\Sigma (x,t,s)^{-1}\cdot p\big) \sqrt{1+(\nabla^\phi \phi(x,t,s))^2}\,dx\,dt\,ds\Big|\\
&\lesssim 1+    \Big|\int_{\mathbb{W}}\Big( K_{I\!I}\big(\Phi_\Sigma (x,t,s)^{-1}\cdot p\big) \sqrt{1+(\nabla^\phi \phi(x,t,s))^2} -K_{I\!I} \big(\Phi_{p\mathbb{W}}(x,t,s)^{-1} \cdot p\big)\Big)\,dx\,dt\,ds\Big|\\
&\leq 1 + \int_\mathbb{W} \big|K_{I\!I}\big(\Phi_\Sigma(x,t,s)^{-1}\cdot p\big)\big|\cdot \Big|\sqrt{1+(\nabla^\phi \phi(x,t,s))^2} -1\Big|\,dx\,dt\,ds\\
&\quad + \int_\mathbb{W} \Big|K_{I\!I}^*\big(p^{-1}\cdot \Phi_\Sigma(x,t,s)\big) -K_{I\!I}^* \big(p^{-1}\cdot \Phi_{p\mathbb{W}}(x,t,s)\big)\Big|\,dx\,dt\,ds.
\end{align*}
Consequently:
\begin{align*}
|S_{I\!I}\chi_R(p)|&\lesssim 1 + \sum_{1\leq 2^{-j} < 2^{-k}} 2^{(Q-1)j} \int_{\Phi_\Sigma(x,t,s)\in B(p, 2^{1-j})\backslash B(p,2^{-2-j})} |\nabla^\phi \phi(x,t,s)|\,dx\,dt\,ds\\
&\quad + \sum_{1\leq 2^{-j} < 2^{-k}} \int_{\Phi_\Sigma(x,t,s)\in B(p, 2^{1-j})\backslash B(p,2^{-2-j})}\frac{\|\Phi_{p\mathbb{W}}(x,t,s)^{-1} \cdot \Phi_\Sigma(x,t,s)\|^{\beta/2}}{\|p^{-1}\cdot \Phi_\Sigma(x,t,s)\|^{Q-1+\beta/2}} \,dx\,dt\,ds\\
&\quad + \sum_{1\leq 2^{-j} < 2^{-k}} \int_{\Phi_{p\mathbb{W}}(x,t,s)\in B(p, 2^{1-j})\backslash B(p,2^{-2-j})}\frac{\|\Phi_{p\mathbb{W}}(x,t,s)^{-1} \cdot \Phi_\Sigma(x,t,s)\|^{\beta/2}}{\|p^{-1}\cdot \Phi_\Sigma(x,t,s)\|^{Q-1+\beta/2}} \,dx\,dt\,ds\\
&\lesssim 1 + \sum_{\substack{1\leq 2^{-j} < 2^{-k}\\ 2^{-3-j} \geq \|p\| \text{ or } 2^{2-j} \leq \|p\|}} 2^{(Q-1)j}\cdot 2^{j\gamma} \cdot 2^{-(Q-1)j} + \sum_{\substack{1\leq 2^{-j} < 2^{-k}\\ 2^{-3-j} <\|p\|, 2^{2-j}> \|p\|}} 2^{(Q-1)j} \cdot 1\cdot 2^{-(Q-1)j}\\
&\quad + \sum_{1\leq 2^{-j} < 2^{-k}} \int_{\Phi_\Sigma(x,t,s)\in B(p, 2^{1-j})\backslash B(p, 2^{-2-j})}\frac{|\phi^{(p^{-1})} (\pi_\mathbb{W}(p^{-1} \cdot \Phi_\Sigma(w)))|^{\beta/2}}{\|p^{-1}\cdot \Phi_\Sigma(x,t,s)\|^{Q-1+\beta/2}} \,dx\,dt\,ds\\
&\quad + \sum_{1\leq 2^{-j} < 2^{-k}} \int_{\Phi_{p\mathbb{W}}(x,t,s)\in B(p, 2^{1-j})\backslash B(p, 2^{-2-j})}\frac{|\phi^{(p^{-1})} (\pi_\mathbb{W}(p^{-1} \cdot \Phi_\Sigma(w)))|^{\beta/2}}{\|p^{-1}\cdot \Phi_\Sigma(x,t,s)\|^{Q-1+\beta/2}} \,dx\,dt\,ds\\
&\lesssim_\gamma 1 +\sum_{1\leq 2^{-j}<2^{-k}} 2^{-j(1-\theta)\beta/2} \cdot 2^{j(Q-1+\beta/2)}\cdot 2^{-(Q-1)j}\\
&= 1 + \sum_{1\leq 2^{-j}< 2^{-k}}2^{j\theta \beta/2} \lesssim_\theta 1.
\end{align*}
where we have used the fact that if $\Phi_{p\mathbb{W}} (x,t,s) \in B(p, 2^{1-j})\backslash B(p, 2^{-2-j})$, then
\begin{align*}
&\quad \big\|p^{-1}\cdot \Phi_\Sigma (x,t,s)\big\| = \big\|\pi_\mathbb{W}(p^{-1}\cdot \Phi_\Sigma (w))\cdot \phi^{(p^{-1})}(\pi_\mathbb{W}(p^{-1}\cdot \Phi_\Sigma (w)))\big\|\\
&=\big\|p^{-1} \cdot \Phi_{p\mathbb{W}} (x,t,s) \cdot \phi^{(p^{-1})}(p^{-1} \cdot \Phi_{p\mathbb{W}} (x,t,s))\big\| \sim 2^{-j},
\end{align*}
and $|\phi^{(p^{-1})}(\pi_\mathbb{W}(p^{-1}\cdot \Phi_\Sigma (w)))|\lesssim \|\pi_\mathbb{W}(p^{-1}\cdot \Phi_\Sigma(w))\|^{1-\theta}\lesssim 2^{-j(1-\theta)}$. Note the estimate used above
\begin{equation}\label{legitimate}
\Big|K_{I\!I}^*\big(p^{-1}\cdot \Phi_\Sigma(x,t,s)\big) -K_{I\!I}^* \big(p^{-1}\cdot \Phi_{p\mathbb{W}}(x,t,s)\big)\Big|\lesssim \frac{\|\Phi_{p\mathbb{W}}(x,t,s)^{-1} \cdot \Phi_\Sigma(x,t,s)\|^{\beta/2}}{\|p^{-1}\cdot \Phi_\Sigma(x,t,s)\|^{Q-1+\beta/2}} 
\end{equation}
is legitimate if $d(z_{1},z_{2}) \leq  \kappa \|z_{1}\|$ (assume without loss of generality that $\|z_1\|\geq \|z_2\|$). The condition $d(z_{1},z_{2}) \leq  \kappa \|z_{1}\|$ may not be the case to begin with, but then we know that $2^{-j} \lesssim \kappa \|z_{1}\| \leq d(z_{1},z_{2}) \lesssim 2^{-j}$, and we can pick boundedly many points $z_{1}',\ldots,z_{n}'$ with the following properties:  $z_{1}' = z_{1}$, $z_{n}' = z_{2}$, $\|z_{i}'\| \sim 2^{-j}$ and $2^{-j} \sim d(z_{i}',z_{i + 1}') \leq\kappa \|z_{i}'\|$; in particular $d(z_{i}',z_{i + 1}') \lesssim d(z_{1},z_{2})$. Then, we obtain (\ref{legitimate}) by the triangle inequality, and boundedly many applications of the H\"older estimates.

\begin{remark} Let $\G$ be a step-2 Carnot group equipped with a strongly homogeneous norm $\|\cdot\|$ and let $\mathbb{W}$ be a vertical hyperplane. Let $\alpha>0$ and let $\phi\in C^{1, \alpha}(\mathbb{W})$: i.e. there exists some $H>0$ such that
$$\|d\phi^{(p^{-1})}_w - d\phi^{(p^{-1})}_\0\|_{\text{op}} \leq H \|w\|^\alpha, \qquad \forall p \in \Sigma, w \in \bW,$$
where $\Sigma$ is the intrinsic graph of $\Sigma$. Furthermore, assume that there exist $\gamma>0, \theta \in (0,1),$  such that for every $w=(x,z)\in \mathbb{W}, p\in \Sigma$,
\begin{equation*}
|\phi^{(p^{-1})}(x,z)| \leq C_\theta \|(x,z)\|^{1-\theta}, \qquad |\nabla^\phi \phi (x,z)|\leq C_\gamma \|(x,z)\|^{-\gamma},
\end{equation*}
where $C_\gamma, C_\theta \geq 1$. Let $\mu$ be a $(Q-1)$-ADR measure with constant $C_\mu$, which is supported in $\Sigma$. Let $K$ be a $(Q-1)$-dimensional Calder\'on-Zygmund kernel which satisfies \eqref{czgrowth} and \eqref{czhol} with constant $C_K\geq 1$. Also, assume that $K$ and $K^{\ast}$ satisfy the annular boundedness condition with a constant $A$ (this constant depends on $\|\cdot\|$ and a fixed radial function $\psi$). It then follows by the proof of Theorem \ref{theoreminfinity} that the constants in \eqref{restricted} depend only on $H, \theta, \gamma, C_\theta, C_\gamma, C_\mu, C_K$ and $A$. It then follows by the $T1$ theorem that for all $f \in L^2(\mu), \ve>0$,
$$\|T_{\mu, \ve}f\|_{L^2(\mu)} \leq C \|f\|_{L^2(\mu)}$$
where $C$ only depends on the aforementioned constants.

If $K$ is horizontally antisymmetric then, by the proof of Lemma \ref{antisymmetry},  there is no dependance on $A$. Moreover, in the case when $\phi$ is compactly supported instead of the dependence on $\theta, \gamma, C_\theta, C_\gamma$ there is only dependance on the diameter of the $\supp \phi$.
\end{remark}

\section{Removable sets for Lipschitz harmonic functions and the $(Q-1)$-dimensional Riesz transform.}\label{sec:rem}
The goal of this section is to prove Theorem \ref{thm-removintro} and Corollary \ref{riesznonremovable}. 
We start with some definitions.
\begin{definition}\label{def:harm}
Let $\G$ be a Carnot group and let $D\subset \G$ be an open set. A real valued function $f:D\to\R$ is called $\Delta_\G$-harmonic, or simply harmonic, on $D$ 
if it is smooth and $\Delta_\G f= 0$ on $D$ in the distributional sense.
\end{definition}

The smoothness assumption in the previous definition is not essential, as $\Delta_G$-harmonic distributions are smooth due to the classical hypoellipticity theorem of Hörmander.

\begin{definition}
\label{def:remov} 
Let $\G$ be a Carnot group equipped with a homogeneous norm $\|\cdot\|$. A closed set $C \subset \G$ is called \textit{removable for Lipschitz $\Delta_\G$-harmonic functions} (or $\G$-RLH set), if for every domain $D \subset \G$ with every locally Lipschitz function $f:(D,\|\cdot\|) \to \R$
that is $\G$-harmonic in $D\stm C$, it is also $\Delta_\G$-harmonic in $D$.
\footnote{The reader might notice that our definition of $\G$-RLH sets is slightly different than the definition appearing  in \cite{CFO19, CMT} where $C$ has to be contained in the domain $D$. Our definition is a direct adaptation of the definition for Euclidean RLH sets appearing in \cite[Definition 2.1]{MR1372240} and \cite{ntv2}. Definition \ref{def:remov} immediately implies that Carnot RLH sets are monotone, in the sense that if $C, F$ are closed, $F$ is RLH and $C \subset F$, then $C$ is RLH as well. As far as we know, monotonicity is not immediate if one uses the definition from \cite{CFO19, CMT}. However, both definitions are equivalent and this follows from a characterization of RLH sets via a suitable capacity as in \cite{MR1372240}. This will appear in forthcoming work of Boone.}
\end{definition} 
For the following definition recall the $\G$-Riesz kernel (the horizontal gradient of the fundamental solution of the sub-Laplacian $\Delta_\G$) which was defined in Example \ref{griesz}. Recall also Lemma \ref{lem:rieszarecz}.
\begin{definition}
\label{def:riesztransform}
Let $\G$ be a Carnot group equipped with a homogeneous norm $\|\cdot\|$. Let $\mu$ be a positive $(Q-1)$--upper regular measure in $\G$. The truncated $(Q-1)$-dimensional Carnot Riesz transforms (with respect to $\|\cdot\|$) are defined as 
$$\mathcal{R}_{\mu,\varepsilon} f (p)=\int_{\|q^{-1}\cdot p\|>\varepsilon} \mathsf{R}(q^{-1}\cdot p) f(q) d \mu (q),$$
where $\varepsilon>0$, $p \in \G$ and $f \in L^2(\mu)$.
\end{definition}
For the convenience of the reader we restate Theorem \ref{thm-removintro}.
\begin{theorem}
\label{thm-remov}
Let $\G$ be a Carnot group equipped with a homogeneous norm $\|\cdot\|$. Let $\mu$ be a positive $(Q-1)$--upper regular measure in $\G$  such that  $\supp\mu$ has locally finite $(Q-1)$-dimensional Hausdorff measure. If the Riesz transform $\mathcal{R}$ is bounded on $L^2(\mu)$, then $\supp\mu$ is not removable for Lipschitz harmonic functions.
\end{theorem}
The proof of Theorem \ref{thm-removintro} is very similar to the proof of \cite[Theorem 5.1]{CFO19}, which adapted a well-known scheme of Uy \cite{uy} (see also \cite[Theorem 4.4]{MR1372240} and \cite[Chapter 4]{tolsabook}) to the first Heisenberg group. Almost every step in the proof of \cite[Theorem 5.1]{CFO19} can be transferred to our case without problems after some obvious modifications. The only step  which requires a different argument lies in the proof of the following lemma; this is the generalization of \cite[Lemma 5.4]{CFO19} in step-2 Carnot groups.
\begin{lemma}
\label{lem:liphorizgrad}
Let $E\subset \mathbb{G}$ be a set of locally finite $(Q-1)$-dimensional measure and let $f: \mathbb{G}\to \mathbb{R}$ be continuous. If $f\in C^1(\mathbb{G}\backslash E)$ and $\nabla_{\mathbb{G}} f\in L^\infty(\mathbb{G}\backslash E)$, then $f$ is Lipschitz on $\mathbb{G}$.
\end{lemma}
The proof of Lemma \ref{lem:liphorizgrad} employs two auxillary results. First:
\begin{lemma}
\label{lem:1CFO}
Fix a vertical hyperplane $\mathbb{W}$, and let $E\subset \mathbb{G}$. Then
$$
\int_\mathbb{W}^* \mathcal{H}^{s-Q+1} (E\cap \pi_\mathbb{W}^{-1}(w)) \,d\mathcal{L}^{m+n+k-1}(w) \lesssim \mathcal{H}^s(E), \quad \forall s\in[Q-1,Q],
$$
where the notation $\int^*_\mathbb{W}$ means it take value $0$ when it is the integral diverges.
\end{lemma}
The proof of Lemma \ref{lem:1CFO} is identical to the proof of \cite[Lemma 5.3]{CFO19}. We record that the proof in \cite[Lemma 5.3]{CFO19} uses \cite[Lemma 2.2]{Franchi-Serapioni-JGA} which is valid in all Carnot groups.

Second, the proof of Lemma \ref{lem:liphorizgrad} makes use of the fact that Carnot groups are ``horizontally polygonally quasiconvex''. This means that any two points in the group can be connected by segments of \textit{horizontal lines} and the union of these segments has length bounded by a constant multiple of the distance between these two points.

Recall that a \textit{horizontal line} in $\G$ is a left coset of an $1$-dimensional horizontal subgroup of $\G$. Equivalently, horizontal lines can be defined as sets of  the form $\pi_\mathbb{W}^{-1}(w)$ where $\mathbb{W}$ is a vertical hyperplane and $w\in \mathbb{W}$. 

The fact that the first Heisenberg group is horizontally polygonally quasiconvex was proved in \cite[Lemma 5.2]{CFO19}. We will now prove that any Carnot group satisfies the same property.

\begin{lemma}
\label{lem:hpq}
Let $\mathbb{G}$ be a Carnot group of step $s$ whose first layer has dimension $m$. There exist $N:=N(m,s)$ and $C:=C(m,s)$  such that for any $p_1, p_2\in \mathbb{G}$, there exist at most $N$ horizontal line segments $l_i$ with connected union, containing $p_1, p_2$, and 
$$
\sum \mathcal{H}^1(l_j) \leq C d(p_1, p_2).
$$
\end{lemma}

\begin{proof}
For step $1$ group, this is trivial. We first prove for the case where $\mathbb{G}$ is of step $2$. Since all strongly homogeneous metrics are comparable, without loss of generality, we can work with 
$$d(p_1, p_2) = \|p_2^{-1}\cdot p_1\|$$
where $\|\cdot\|$ is the strongly homogeneous norm
$$
\|q\|= |x| + |z|^{1/2},
$$
where $q=(x,z)$ with $x$ denoting the first layer, $z$ denoting the second layer and $|\cdot|$ denoting the Euclidean metric. By left invariance we can assume that $p_1=\zero$. Let $p_2=(x_2, z_2)$. We first connect $\zero$ to $(x_2, 0)$ by a horizontal line segment of length $\leq d(p_1, p_2)$. We will now describe how to connect from $(x_2, 0)$ to $(x_2, z_2)$ with horizontal line segments. By left invariance, it is enough to elaborate how we can connect $(0,0)$ to $(0,z_2)$ using horizontal line segments. Let $n:=\frac{Q-m}{2}$ and write $z_2=(z_{2,1}, \ldots, z_{2,n})$. Since each $Z_j$ is a linear combination of the $[X_i, X_k]$, we have
\begin{equation}\label{squares}
\begin{aligned}
\sum_{j=1}^n z_{2,j}Z_j &= \sum_{j=1}^n z_{2,j} \sum_{i,k=1}^m C_{j,i,k} [X_i, X_k]\\
&=\sum_{i,k=1}^m \big(\sum_{j=1}^n z_{2,j}C_{j,i,k}\big) [X_i, X_k].
\end{aligned}
\end{equation}
For each $(i,k)$, if $\sum_j z_{2,j} C_{j,i,k}>0$, starting from any given point, consider a ``lifted square'', i.e. $4$ connecting horizontal line segments subsequently traveling with unit time along 
$$
\sqrt{\big|\sum_j z_{2,j}C_{j,i,k}\big|} X_i, \quad \sqrt{\big|\sum_j z_{2,j}C_{j,i,k}\big|} X_k, \quad -\sqrt{\big|\sum_j z_{2,j}C_{j,i,k}\big|} X_i, \quad -\sqrt{\big|\sum_j z_{2,j}C_{j,i,k}\big|} X_k.
$$
If $\sum_j z_{2,j} C_{j,i,k}<0$, starting from any given point, consider a lifted square with $4$ sides subsequently traveling with unit time along 
$$
\sqrt{\big|\sum_j z_{2,j}C_{j,i,k}\big|} X_k, \quad \sqrt{\big|\sum_j z_{2,j}C_{j,i,k}\big|} X_i, \quad -\sqrt{\big|\sum_j z_{2,j}C_{j,i,k}\big|} X_k, \quad -\sqrt{\big|\sum_j z_{2,j}C_{j,i,k}\big|} X_i.
$$
Each lifted square has total length of sides $\lesssim |z_2|^{1/2}\lesssim d(p_1, p_2)$. To connect $(0,0)$ to $(0, z_2)$ with horizontal line segments, starting from $(0,0)$, we travel along the lifted squares one by one, each corresponding to a nonzero term $(\sum_j z_{2,j}C_{j,i,k})[X_i, X_k]$ in (\ref{squares}); the order of the lifted squares does not matter. Since there are $\binom{m}{2}$ nonzero terms in (\ref{squares}), there are $\binom{m}{2}$ lifted squares and the total length of line segments $\lesssim d(p_1,p_2)+4\binom{m}{2} d(p_1,p_2) \lesssim d(p_1, p_2)$. 

For $\mathbb{G}$ of step $s>2$, we need to connect from $p_1=(0,0,\ldots, 0)$ to $p_2=(w_1,w_2, \ldots, w_s)$, where the $w_i$ are $i$-th layer coordinates. We first connect from $(0,0,\cdots, 0)$ to $(w_1, 0, \ldots, 0)$, which needs $1$ horizontal line with length $=|w_1|\lesssim \|w\|=d(p_1, p_2)$. We then connect from $(w_1, 0,0,\ldots, 0)$ to $(w_1, 0,0,\ldots, 0)\cdot(0, w_2,0, \ldots, 0)=(w_1, w_2, w_3',0,\ldots,0)$, where $|w_3'-w_3|\lesssim d(p_1, p_2)^3$. (Similar as above in the step $2$ case, the $[X_i, X_k]$ correspond to squares, the $[[X_i, X_k], X_j]$ correspond to decagons, etc.) This needs no more than $\binom{m}{2}$ squares and thus $4\binom{m}{2} \leq 4m^2$ horizontal lines, with total length $\lesssim d(p_1, p_2)$. We then connect from $(w_1, w_2, w_3', 0,\ldots, 0)$ to $(w_1, w_2, w_3', 0,\ldots, 0)\cdot (0, 0, w_3-w_3', 0,\ldots, 0) =(w_1, w_2, w_3, w_4', w_5', 0,\ldots, 0)$, where $|w_4'-w_4|\lesssim d(p_1, p_2)^4, |w_5'-w_5|\lesssim d(p_1, p_2)^5$ (if there are $4$th and $5$th layers). This needs no more than $m^3$ decagons and thus $10 m^3$ horizontal lines, with total length $\lesssim d(p_1,p_2)$. We then continue similarly until we reach $(w_1, \ldots, w_s)$. Thus for the Carnot group $\mathbb{G}$ of step $s$, if the first layer of $\mathbb{G}$ is of dimension $m$, the number of connected horizontal line segments $l_i$ needed to contain $p_1$ and $p_2$ is no more than $1+ 4m^2 + 10m^3+ 22m^4+\cdots +(3\cdot 2^{s-1}-2)m^s$, and $\sum \mathcal{H}^1(l_j) \lesssim d(p_1, p_2)$.
\end{proof}

With Lemmas \ref{lem:1CFO} and \ref{lem:hpq} at hand, the proof of Lemma \ref{lem:liphorizgrad} follows exactly as in the proof of Lemma \cite[Lemma 5.4]{CFO19}.

We will now discuss the proof of Theorem \ref{thm-remov}. Having Lemma \ref{lem:liphorizgrad} at our disposal, the proof of Theorem \ref{thm-remov} proceeds in the same way (with obvious modifications) as in the proof of \cite[Theorem 5.1]{CFO19}. We will only provide an outline for convenience of the reader.
\begin{proof}[Outline of the proof of Theorem \ref{thm-remov}] Consider the coordinate truncated $\G$-Riesz transforms:
$$\mathcal{R}^i_{\mu,\varepsilon} f(p)=\int_{\|q^{-1}\cdot p\|>\varepsilon} \mathsf{R}^i(q^{-1}\cdot p) f(q) d \mu (q), \quad p \in \G,f \in L^2(\mu)\; i=1,\dots,n$$
where $\mathsf{R}(p):=(\mathsf{R}^1(p), \dots, \mathsf{R}^n(p))=\nabla_{\G} \Sigma(p)$. We will also need their smoothened versions. Let $\phi:\R \to [0,1]$ be a $C^{\infty}$ function such that $\phi|_{(-1/2,1/2)}=0$ and $\phi|_{\R \stm (-1,1)}=1$. Let
\begin{displaymath}
\widetilde{\mathcal{R}}^i_{\mu,\varepsilon}(f)(p)=\int \phi \left( \frac{\|q^{-1}\cdot p\|}{\varepsilon}\right)\mathsf{R}^i(q^{-1}\cdot p)f(q) \, d\mu(q),\quad i=1,\dots,n, \quad f \in L^2(\mu).
\end{displaymath}
Moreover, if $\nu$ is a signed Radon measure in $\G$ we define
$$\widetilde{\mathcal{R}}^i_{\varepsilon}\nu(p)=\int \phi \left( \frac{\|q^{-1}\cdot p\|}{\varepsilon}\right)\mathsf{R}^i(q^{-1}\cdot p) \, d\nu(q),\quad i=1,\dots,n.$$
We let $\mathcal{M}(\G),$ be the space of  signed Radon measures on $\G$  with finite total variation $\|\nu\|_{TV}$. It is well known that  $(\mathcal{M}(\G),\|\cdot\|_{TV})$ is the dual of $(C_0(\G), \|\cdot\|_\infty)$ where $C_0(\G)=C_0(\R^N)$ denotes the vector space of continuous functions which vanish at infinity.


Since, by assumption, the operators $\mathcal{R}^i_{\mu,\varepsilon},i=1,\dots,n,$ are uniformly bounded in $L^2(\mu)$, \eqref{eq:adjoint} implies that their adjoints $\mathcal{R}^{i,\ast}_{\mu,\varepsilon},i=1,\dots,n,$ are uniformly bounded in $L^2(\mu)$ as well. Using this fact and a standard comparison argument involving the maximal function, it's not difficult to show that the adjoints of the smoothened coordinate Riesz transforms $\widetilde{\mathcal{R}}^{i,\ast}_{\mu,\varepsilon}$ are also uniformly bounded in $L^2(\mu)$.
By a result of Nazarov, Treil and Volberg \cite[Corollary 9.2]{MR1626935}, we can then conclude that the operators $\tilde{\mathcal{R}}^{i,\ast}_{\ve}: \mathcal{M}(\G) \to L^{1,\infty}(\mu)$
are uniformly bounded; i.e. for every $\lambda>0$ and every $\nu \in \mathcal{M}(\G)$,
\begin{equation}
\label{eq:weak11}
\mu (\{p \in \G: |\tilde{\mathcal{R}}^{i,\ast}_{\ve}\nu (p)|> \lambda\}) \leq C \frac{\|\nu\|_{TV}}{\lambda}, \quad i=1,\dots,n,
\end{equation}
where $C$ is an absolute constant.

Moreover, 
\begin{equation}
\label{eq:smoothendriesz}
\widetilde{\mathcal{R}}^i_{\varepsilon}:\mathcal{M}(\G)\to \mathcal{C}_0(\G)\quad \mbox{ and } \quad\widetilde{\mathcal{R}}^{i,\ast}_{\varepsilon}:\mathcal{M}(\G)\to \mathcal{C}_0(\G).
\end{equation}
This is the reason why smoothened versions are needed, since the operators $\mathcal{R}^i_{\varepsilon}$ do not map $\mathcal{M}(\G)$ to $\mathcal{C}_0(\G)$.

From now on we fix a compact  $E \subset \supp \mu$  with $0<\mu(E)<\infty$. With \eqref{eq:weak11} and \eqref{eq:smoothendriesz} at our disposal, we can apply a well known theorem due to Davie and \O{}ksendal \cite{DO} (see also \cite[Lemma 4.2]{MR1372240} or \cite[Theorem 4.6]{tolsabook}) to the smoothened operators $\mathcal{R}^i_{\varepsilon}$ and obtain, for every $\ve>0$,  a function $h_{\ve} \in L^{\infty}(\mu)$ which satisfies
\begin{enumerate}
\item $0 \leq h_{\ve}(p) \leq \chi_{E}(p)$ for $\mu$ almost every $p \in \G$, 
\item \label{eq:hemeas}
$\int h_{\ve} \, d\mu \geq \mu(E)/2$
\item \label{eq:linfbound}
$\|\tilde{\mathcal{R}}^i_{\mu,\varepsilon} h_{\ve}\|_\infty \lesssim 1.$
\end{enumerate}

We now consider the functions
$$f_{\ve}(p)=\int \Sigma(q^{-1} \cdot p) h_{\ve}(q) \, d\mu(q), \quad p \in \G.$$
By the left invariance of $\nabla_{\G}$:
$$\nabla_{\G} f_{\ve} (p)=(\widetilde{\mathcal{R}}^i_{\mu,\varepsilon} h_{\ve}(p),\dots,\widetilde{\mathcal{R}}^2_{\mu,\varepsilon} h_{\ve}(p))$$
for $p \in \G$ such that $d(p,E)\geq \ve$. Thus, for such $p$, \eqref{eq:linfbound} implies
\begin{equation}
\label{eq:nablabound}
|\nabla_{\G} f_{\ve} (p)| \lesssim 1
\end{equation}

By the Banach-Alaoglu Theorem, see  \cite[Corollary 3.30]{brezis}, there exists a sequence $\ve_n \to 0$ and a function $h$ with $\|h\|_{L^\infty(\mu|_{E})} \leq 1$ such that
\begin{equation}
\label{eq:weakstar}\int_E h_{\ve_n}g\, d\mu \to \int_E hg\, d\mu, \quad \mbox{ for all }g \in L^1(\mu|_{E}).
\end{equation}
By \eqref{eq:weakstar} and \eqref{eq:hemeas} we see that
\begin{equation}
\label{eq:lastbound}
\int_E h \, d\mu \geq \mu(E)/2.
\end{equation}

We now set $\nu=h \,d\mu$. Note that  $\supp \nu \subset E$ (assuming that $h=0$ outside $E$) and $\nu$ is $(Q-1)$--upper regular. Let
$$f(p) :=\int \Sigma(q^{-1}\cdot p) d\nu(q), \qquad p \in \G.$$
We will show that $f$ is Lipschitz,  harmonic in $\G \stm E$  but not harmonic in $\G$. This will establish that $E$, and thus $\supp \mu$ is not removable.


First, by the equicontinuity of $(f_{\ve_n})$ 
on compact subsets of $\G \stm E$, we can apply the Arzel\`a-Ascoli theorem
to get a sequence $(f_{\ve_{n_l}})$  which converges uniformly
on compact subsets of $\G \stm E$. Thus, by the Mean Value Theorem for
sub-Laplacians and its converse, see \cite[Theorem 5.5.4 and Theorem 5.6.3]{BLU},  we deduce that that $f$ is harmonic in $\G \stm E$.

We will use Lemma \ref{lem:liphorizgrad} in order to show that $f$ is Lipschitz in $\G$. First, a standard argument is used to show that $f$ is continuous, see \cite[]{CFO19}. Since $f$ is harmonic in $\G \stm E$, H\"ormander's theorem \cite[Theorem 1]{BLU}  implies that  $f \in C^\infty (\G \setminus E)$. Thus it suffices to show that $\nabla_{\G} f \in L^{\infty}(\G \setminus E)$. This is achieved by an application of \cite[Proposition 2.4]{CMT}, as in \cite[pages]{CFO19}

It remains to show that $f$ is {\em not} harmonic in $\G$. Recalling \cite[Definition 5.3.1 (iii)]{BLU} we deduce that
$$\langle \Delta_{\G} f, 1\rangle=-\int h\, d\mu \leq -\, \mu(E)/2<0.$$
Thus, $f$ is not harmonic in $\G$.
\end{proof}
We now briefly discuss  Corollaries \ref{riesznonremovable} and \ref{cororem2}. For Corollary \ref{riesznonremovable} let $\phi \in C^{1,\alpha}(\mathbb{W})$ with $\alpha > 0$, and let $\Sigma$ be the intrinsic graph of $\phi$. Let $E \subset \Sigma$ be a closed set with $\mathcal{H}^{Q-1}(E) > 0$. By Remark \ref{remk:c1alip} and the fact that the restriction of $\mathcal{H}^{Q-1}$ on intrinsic Lipschitz graphs is $(Q-1)$-ADR (see \cite[Theorem 3.9]{FS}) we deduce that   $\mu = \mathcal{H}^{Q-1}|_{E}$ is $(Q-1)$-upper regular. Therefore, Theorem \ref{thm-remov} and Corollary \ref{maincoro} imply that $E$ is not removable for Lipschitz harmonic functions. Corollary \ref{cororem2} follows as in \cite[Corollary 5.7]{CFO19} and we skip the details.



\bibliographystyle{amsalpha}
\bibliography{ref}
\end{document}